\documentclass[twoside,11pt,reqno]{amsart}
\usepackage{amsmath,amssymb,amscd,mathrsfs,amscd}
\usepackage{graphics,verbatim}
\usepackage{todonotes}
\usepackage{hyperref}

\let\oldsection=\section

\newcommand{\losemi}{{\otimes \kern -.78em \ltimes}}
\newcommand{\rosemi}{{\otimes \kern -.78em \rtimes}}

\newcommand{\Hom}{\ensuremath{\operatorname{Hom}}}

\newcommand{\Ind}{\ensuremath{\operatorname{ind}}}

\newcommand{\Ext}{\operatorname{Ext}}

\newcommand{\0}{\bar 0}
\newcommand{\1}{\bar 1}
\newcommand{\Z}{\mathbb{Z}}
\newcommand{\C}{\mathbb{C}}

\newcommand{\gl}{\ensuremath{\mathfrak{gl}}}
\newcommand{\g}{\ensuremath{\mathfrak{g}}}
\newcommand{\e}{\ensuremath{\mathfrak{e}}}
\newcommand{\f}{\ensuremath{\mathfrak{f}}}

\newcommand{\res}{\ensuremath{\operatorname{res}}}

\newcommand{\X}{\mathcal{X}}

\newcommand{\fa}{\ensuremath{\mathfrak{a}}}
\newcommand{\fg}{\ensuremath{\mathfrak{g}}}
\newcommand{\fb}{\ensuremath{\mathfrak{b}}}
\newcommand{\fu}{\ensuremath{\mathfrak{u}}}

\newcommand{\ff}{\ensuremath{\f}}
\newcommand{\fe}{\ensuremath{\e}}

\newcommand{\fl}{\ensuremath{\mathfrak{l}}}
\newcommand{\fq}{\ensuremath{\mathfrak{q}}}
\newcommand{\fp}{\ensuremath{\mathfrak{p}}}
\newcommand{\ft}{\ensuremath{\mathfrak{t}}}
\newcommand{\fz}{\ensuremath{\mathfrak{z}}}

\renewcommand{\a}{\alpha}

\newcommand{\V}{\mathcal{V}}
\newcommand{\HH}{\operatorname{H}}
\newcommand{\F}{\mathcal{F}}

\newcommand{\la}{\lambda}

\newcommand{\glmn}{\mathfrak{gl}(m|n)}

\newcommand{\R}{\ensuremath{\mathbb{R}}}

\newcommand{\CC}{\mathcal{C}}
\newcommand{\FF}{\mathcal{F}}

\newcommand{\VV}{\mathcal{V}}

\newcommand{\XX}{\mathcal{X}}
\newcommand{\YY}{\mathcal{Y}}

\newcommand{\Tensor}{\operatorname{Tensor}}
\newcommand{\Loc}{\operatorname{Loc}}
\newcommand{\Proj}{\operatorname{Proj}}

\newcommand{\bT}{\mathbf T}
\newcommand{\bS}{\mathbf S}
\newcommand{\bC}{\mathbf C}
\newcommand{\bK}{\mathbf K}
\newcommand{\bL}{\mathbf L}
\newcommand{\bP}{\mathbf P}
\newcommand{\bI}{\mathbf I}
\renewcommand{\Im}{\operatorname{Im}}
\newcommand{\unit}{\ensuremath{\mathbf 1}}
\renewcommand{\Gamma}{\varGamma}

\newtheorem{theorem}{Theorem}[subsection]

\makeatletter\let\c@fact\c@theorem\makeatother

\makeatletter\let\c@note\c@theorem\makeatother

\newtheorem{lemma}{Lemma}[subsection]
\makeatletter\let\c@lemma\c@theorem\makeatother

\makeatletter\let\c@alg\c@theorem\makeatother

\makeatletter\let\c@remark\c@theorem\makeatother

\newtheorem{prop}{Proposition}[subsection]
\makeatletter\let\c@prop\c@theorem\makeatother

\makeatletter\let\c@conj\c@theorem\makeatother

\makeatletter\let\c@cor\c@theorem\makeatother

\newtheorem{defn}{Definition}[subsection]
\makeatletter\let\c@defn\c@theorem\makeatother

\theoremstyle{definition}
\newtheorem{example}{Example}[subsection]
\makeatletter\let\c@example\c@theorem\makeatother
\numberwithin{equation}{subsection}

\usepackage[capitalise]{cleveref}
\crefname{theorem}{Theorem}{Theorems}
\crefname{fact}{Fact}{Facts}
\crefname{note}{Note}{Notes}
\crefname{lemma}{Lemma}{Lemmas}
\crefname{alg}{Algorithm}{Algorithms}
\crefname{remark}{Remark}{Remarks}
\crefname{example}{Example}{Examples}
\crefname{prop}{Proposition}{Propositions}
\crefname{conj}{Conjecture}{Conjectures}
\crefname{cor}{Corollary}{Corollaries}
\crefname{defn}{Definition}{Definitions}
\crefname{equation}{\!\!}{\!\!} %Remove spacing around phantom equation name

\newcounter{listequation}
\newenvironment{eqnlist}{\begin{list}
{(\thesubsection.\thelistequation)}
{\usecounter{listequation} \setlength{\itemsep}{1.0ex plus.2ex
minus.1ex}}\setcounter{listequation}{\value{equation}}}
{\setcounter{equation}{\value{listequation}}\end{list}}

\begin{document}
\title{Tensor Triangular Geometry for Classical Lie Superalgebras}

\author{Brian D. Boe }
\address{Department of Mathematics \\
          University of Georgia \\
          Athens, GA 30602}
\email{brian@math.uga.edu}
\author{Jonathan R. Kujawa}
\address{Department of Mathematics \\
          University of Oklahoma \\
          Norman, OK 73019}
\thanks{Research of the second author was partially supported by NSF grant
DMS-1160763 and NSA grant H98230-11-1-0127}\
\email{kujawa@math.ou.edu}
\author{Daniel K. Nakano}
\address{Department of Mathematics \\
          University of Georgia \\
          Athens, GA 30602}
\thanks{Research of the third author was partially supported by NSF
grants  DMS-1402271 and DMS-1701768}\
\email{nakano@math.uga.edu}
\date{\today}

\begin{abstract} Tensor triangular geometry as introduced by Balmer \cite{balmer} is a powerful idea which can be used to extract the ambient geometry from a given tensor triangulated category. In this paper we provide a general setting for a compactly generated tensor triangulated category which enables one to classify thick tensor ideals and the Balmer spectrum. For the general linear Lie superalgebra $\fg=\fg_{\0}\oplus \fg_{\1}$ we construct a Zariski space from a detecting subalgebra of $\fg$ and demonstrate that this topological space governs the tensor triangular geometry for the category of finite dimensional $\fg$-modules which are semisimple over $\fg_{\0}$. 
\end{abstract}

\maketitle

\section{Introduction}\label{S:intro}   

\subsection{} A predominant theme in representation theory is the utilization of ambient geometric structures to study
a given module category. For a symmetric monoidal tensor triangulated category, $\bK$, Balmer \cite{balmer} first introduced the idea of tensor triangular geometry and used it to show that the underlying geometry can be revealed through the use of the tensor structure. He defined the notion of a prime ideal and constructed the (Balmer) spectrum, $\operatorname{Spc}({\bK})$, thus allowing one to study these categories from the viewpoint of commutative algebra and algebraic geometry.  In particular he showed that a scheme can be reconstructed from an associated tensor triangulated category via his construction.  Another important example occurs when $\bK=\operatorname{Stab}(\text{mod}(G))$ is the stable module category of finitely generated modules for a finite group scheme $G$ over a field $k$. In this setting there exists a homeomorphism between $\operatorname{Spc}({\bK})$ and 
$\Proj( R):=\operatorname{Proj}(\operatorname{Spec}(R))$ where $R=\HH^{2\bullet}(G,k)$ is the cohomology ring for $G$. There are typically many support variety theories (support data) for $\bK$, with $\operatorname{Spc}({\bK })$ being the universal (or ``final'') support variety theory. Determining $\operatorname{Spc}({\bK})$ is closely connected to the classification of thick tensor ideals and recovers geometry hidden within ${\bK}$.  Computing the spectrum of interesting tensor triangulated categories remains one of the most fundamental questions in the subject.  See \cite{balmerICM} for further discussion of the spectrum.

Methods for classifying thick tensor ideals originated in the work of Hopkins \cite{Hop} in the context of the derived category of bounded complexes of finitely generated projective modules over a commutative Noetherian ring. Benson, Carlson, and Rickard \cite{BCR} later studied this question for the stable module category 
of a group algebra. In their work it became apparent that infinitely generated modules would play a key role in the classification.  Specifically, they provide a systematic treatment using ideas from homotopy theory and Rickard's idempotent modules to demonstrate that the thick tensor ideals for the stable module category of a finite group are in correspondence with the specialization closed sets in $\operatorname{Proj}(R)$. The description of $\operatorname{Spc}({\bK })$ can then be deduced from this fact.  In \cite{FPe:07} Friedlander and Pevtsova introduced a new approach using so-called $\pi$-points and extended the above description of $\operatorname{Spc}({\bK })$ to arbitrary finite group schemes.   A fundamental new idea recently introduced by Benson, Iyengar and Krause \cite{BIK:11, BIK:12} is that of having a commutative ring $R$ stratify the category $\bK$. This allowed them to further develop the theory while recovering the aforementioned results in the case of finite groups.

\subsection{} The authors of this paper initiated a study of classical Lie superalgebras via cohomology and support varieties in \cite{BKN1,BKN2, BKN3, BKN4}. 
Given a classical Lie superalgebra $\fg=\fg_{\0}\oplus \fg_{\1}$ over $\C$ let $G_{\0}$ be the connected reductive algebraic group with $\operatorname{Lie}\left( G_{\0}\right)=\fg_{\0}$.  It is natural to consider the category $\FF=\FF_{(\fg,\fg_{\0})}$ of finite dimensional $\fg$-modules which admit a compatible action by $G_{\0}$ and are completely reducible as $G_{\0}$-modules. The category $\FF$ enjoys many of the features found for finite group schemes except the blocks in $\FF$ can have infinitely many 
irreducible representations. One can use the fact that $\FF$ is self-injective 
together with the coproduct and counit in the enveloping algebra $U(\fg)$ to prove that $\bK=\operatorname{Stab}(\FF)$ is a tensor triangulated category.  While many aspects of the category $\FF$ are reasonably well understood, the tensor structure remains elusive (see \cite{BruCatO}); thus a natural problem is to classify the thick tensor ideals and to compute the Balmer spectrum for $\bK$.  

The cohomology ring $R=\HH^{\bullet}(\fg,\fg_{\0};\C)$ identifies with $S^{\bullet}(\fg_{\1}^{*})^{G_{\0}}$, and is thus finitely generated. 
Using this fact one can construct a cohomological support variety $V_{(\fg,\fg_{\0})}(M)$ for $M\in \FF$. 
In \cite{BKN1} using classical invariant theory we constructed two types of (classical) detecting Lie subsuperalgebras of $\fg$, denoted by $\fe$ and 
$\ff$, which can be chosen so that $\fe\leq \ff \leq \fg$. These subalgebras have the striking property that they detect the cohomology ring for $\FF$; that is, the restriction maps 
$$\HH^{\bullet}(\fg,\fg_{\0};\C)\stackrel{\res}\longrightarrow
\HH^{\bullet}(\ff,\ff_{\0};\C)^{N}\stackrel{\res}\longrightarrow \HH^{\bullet}(\fe,\fe_{\0};\C)^{W}$$ 
 yield isomorphisms as rings. Here $N$ is a non-connected reductive group and $W$ is a finite pseudoreflection group. 

For Type I  classical Lie superalgebras, Lehrer, Nakano and Zhang \cite{LNZ} proved a remarkable fact about the detecting subalgebra $\ff$. For 
$M\in \FF$, the restriction maps in cohomology, 
$\operatorname{res}:\HH^{n}(\fg,\fg_{\0};M)\to \HH^{n}(\ff,\ff_{\0};M)$ 
are  monomorphisms for all $n\geq 0$. Two consequences of this theorem are (i) given $M\in \FF$
the restriction map in cohomology induces the following isomorphisms of support varieties, 
$${V}_{(\fe,\fe_{\0})}(M)/\!/W\rightarrow {V}_{(\ff,\ff_{\0})}(M)/\!/N\rightarrow {V}_{(\fg,\fg_{\0})}(M),$$
and (ii) from these isomorphisms one can prove that $V_{(\fg,\fg_{\0})}(-)$ is a support data (see \cref{ex:cohosupport}). 

 It has been a longstanding question to find a natural geometric object for a given Lie superalgebra which governs the representation theory in a similar way as the nilpotent cone controls the representation theory for Lie algebras and quantum groups. The main goal of this paper is to reveal the ambient geometry for classical Lie superalgebras given by the Balmer spectrum. In particular, a complete picture will be given for $\fg=\mathfrak{gl}(m|n)$ through the use of the aforementioned invariant theory and detecting subalgebras. 

We note that for $\glmn$ there is a consequence of determining the thick tensor ideals which one does not have in the previously studied settings.  Namely, the 
Grothendieck group of $\FF_{(\fg,\fg_{\0})}$ is a module for $\mathfrak{gl}(\infty)$ where the action is by translation functors \cite{brundan1}.  The thick tensor ideals then yield canonical submodules in the Grothendieck group.  
    
\subsection{}       
In earlier work Goetz, Quella and Schomerus \cite{GQS} in the case of $\mathfrak {gl}(1|1)$ and $\mathfrak{sl}(2|1)$ study the classification of finite dimensional indecomposable representations and the direct computation of their tensor products.  From their calculations one can readily determine the thick tensor ideals and the Balmer spectrum of $\FF$ in this case.  For the general case of $\gl (m|n)$ when $\text{min}(m,n)\geq 2$ these computational methods are not feasible because the category $\FF$ has wild representation type.  

Furthermore, unlike in the work of \cite{BIK:12}, the cohomology ring of $\FF$ fails to stratify $\bK=\operatorname{Stab}(\FF)$ and indeed the spectrum of the cohomology ring fails to provide the Balmer spectrum (see \cref{ex:cohosupport}).   
In particular, we do not know of a commutative ring which stratifies $\bK$. This forces us to adopt a different setup in order to classify the thick tensor ideals and compute the Balmer spectrum.  We anticipate that our 
approach will be useful in other contexts where a stratifying ring is not readily available.

The classification of thick tensor ideals for classical Lie superalgebras entails establishing powerful theoretical and computational techniques.  In \cref{S:preliminaries} we follow the ideas of \cite{BIK:12, balmer} by introducing and relating three main inputs: (i) a compactly generated tensor triangulated category (TTC) $\bK$, (ii) a Zariski topological space $X$ and (iii) a support data $V$ which takes compact objects (i.e., ones in $\bK^{c}$) to closed sets in $X$. In \cref{S:classificationtheorems}  with these inputs we establish the general machinery for the classification of thick tensor ideals. It is shown that under suitable conditions one can use $V$ and $X$ to classify thick tensor ideals in $\bK^{c}$ and compute the Balmer spectrum $\operatorname{Spc}(\bK^{c})$.   As in \cite{BCR} and \cite{BIK:12} the driving forces behind the scenes are the localization functors constructed for each thick tensor ideal of the subcategory of compact objects.  Our classification theorem was proven in somewhat greater generality by Dell'Ambrogio \cite{Dell:10} (and was also announced by Pevtsova and Smith \cite{PS}).  For the convenience of the reader we include a condensed discussion of these classification results in order to introduce our conventions and notation, and to keep the paper self-contained.

In \cref{S:classicalLiesuperalgebras} we examine the representation theory of classical Lie superalgebras.   We first show that $\bK^{c}=\operatorname{Stab}(\FF)$  is the set of compact objects in  a certain compactly generated TTC, $\bK$. Examples are given in a number of cases where one can construct 
a final support data $V$ and calculate the thick tensor ideals and $\operatorname{Spc}(\bK^{c})$. These examples include the situation when $\fg$ is a detecting subalgebra $\fe$ or $\ff$, and the case  $(\fq^{+},\fg_{\0})$ where $\fq^{+}$ is the standard parabolic subalgebra of $\mathfrak{gl}(m|n)$. Furthermore, in \cref{ex:cohosupport} we show for $\fg=\mathfrak{gl}(m|n)$ that the cohomological support $V_{(\fg,\fg_{\0})}(-)$  is a support data but does not contain enough information to classify thick tensor ideals.  In particular this shows that the cohomology ring fails to stratify $\bK$.   The main reason for this is that the cohomological support data does not detect projectivity for objects of $\FF_{(\fg,\fg_{\0})}$.  We also show the associated varieties introduced by Duflo and Serganova \cite{dufloserganova} provide a support data but again fail to classify thick tensor ideals (see \cref{ex:dufloserganova}).   
 
In \cref{S:MainTheorem}, we focus on the case  $\fg=\mathfrak{gl}(m|n)$. By using the detecting subalgebra $\ff$ we construct a new support data 
using the cohomological spectrum of $\HH^{\bullet}(\ff,\ff_{\0};\C)$ which, thanks to a result of \cite{LNZ}, does detect projectivity in $\FF_{(\fg,\fg_{\0})}$. We then show that to classify thick tensor ideals with this support data, it suffices to realize the $N$-stable closed conical subvarieties of 
$\ff_{\1}$ as supports of modules in $\FF_{(\fg,\fg_{\0})}$. Once this realization result is established one obtains 
the classification via specialization closed $N$-stable conical subsets of $\ff_{\1}$ (or equivalently, specialization closed subvarieties of $N\text{-}\Proj(S^{\bullet}(\f_{\1}^{*})):=   \operatorname{Proj}(N\text{-Spec}(S^{\bullet}(\f_{\1}^{*})))$).  This in turn implies we have an explicit homeomorphism 
\[
f: N\text{-}\Proj(S^{\bullet}(\f_{\1}^{*}))\xrightarrow{\simeq} \operatorname{Spc}\left(\operatorname{Stab}\left(\FF_{(\fg,\fg_{\0})} \right) \right).
\]  The analogues of the above results are also true for the type $C$ Lie superalgebras $\mathfrak{osp}(2|2n)$.

Finally, in \cref{S:GeometricInduction} and \cref{S:VarietiesandGeometricInduction} we study the structure of modules obtained by geometric induction from a parabolic subalgebra and use this information to prove the necessary realization result used in  \cref{S:MainTheorem}.  In previously considered settings the question of realization was not a major obstacle because one can use standard constructions.  For example, for finite group schemes one can use Carlson modules to realize closed conical sets in the spectrum of the cohomology ring. However, in our case we are looking at 
closed conical sets in $\ff_{\1}$ but need to realize these using $\fg$-modules.  To address this problem we 
apply the machinery of geometric induction introduced by Penkov and Serganova to construct the required modules. Our analysis entails proving new results involving 
geometric induction and support varieties, examining carefully the group action of $N$ on $\ff_{\1}$, and 
applying intricate computational properties of the induction functor in order to solve this problem. 

\subsection{Acknowledgments} The authors would like to thank the referees of this paper. Their insightful comments and suggestions led us to implement significant changes to an earlier version of this manuscript. 

\section{Preliminaries} \label{S:preliminaries}
 
\subsection{Triangulated Categories and Compactness} Let $\bT$ be a triangulated category; then $\bT$ is additive and has 
an equivalence $\Sigma:\bT\rightarrow \bT$ called the shift. Moreover, $\bT$ is equipped with a set of distinguished triangles: 
$$M\rightarrow N \rightarrow Q \rightarrow \Sigma M$$ 
which satisfy the axioms as described in, for example, \cite[Section 1.3]{BIK:12} or \cite{Nee:01}.   

If $\bT$ is a triangulated category, an additive subcategory $\bS$ of $\bT$ is a {\em triangulated subcategory} if (i) $\bS$ is non-empty and full, 
(ii) for $M\in \bS$, $\Sigma^{n}M \in \bS$ for all $n\in {\mathbb Z}$, and (iii) if $M\rightarrow N \rightarrow Q \rightarrow \Sigma M$ is 
distinguished triangle in $\bT$ and if two objects in $\{M,N,Q\}$ are in $\bS$ then the third is in $\bS$. A triangulated subcategory 
$\bS$ is called {\em thick} if $M=M_{1}\oplus M_{2}$ in $ \bS$ implies $M_{j}\in \bS$ for $j=1,2$ (i.e., $\bS$ is closed under 
taking direct summands). For $\CC$  a collection of objects of $\bT$, define $\text{Thick}(\CC)$ as the smallest thick subcategory containing $\CC$.  

Assume that the triangulated category $\bT$ admits set indexed coproducts. 
A {\em localizing subcategory} $\bS$  of $\bT$ is a  triangulated subcategory which is closed under taking set indexed coproducts. Using a version of the Eilenberg swindle one can show that localizing subcategories are necessarily thick \cite{BIK:11}. For $\CC$ a collection of objects of $\bT$, let $\Loc_{\bT}(\CC)=\Loc (\CC )$ 
be the smallest localizing subcategory containing $\CC$. 

An object $C$ in $\bT$ is \emph{compact} if $\Hom_{\bT}(C,-)$ commutes with set indexed coproducts. Let $\bT^{c}$ denoted the full subcategory of compact objects in $\bT$. 
The triangulated category $\bT$ is \emph{compactly generated} if the isomorphism classes of compact objects form a set and if for each non-zero $M \in \bT$ there is a compact object $C$ such that $\Hom_{\bT }(C,M) \neq 0$.  It follows from \cite[Proposition 1.47]{BIK:12} that when $\bT$ is compactly generated one can find a set of compact objects, $\CC$, such that $\Loc_{\bT}(\CC)=\bT$.

\subsection{Tensor Triangulated Categories}\label{SS:TTC} For the purposes of this paper we will work with tensor triangulated categories as defined in \cite[Definition 1.1]{balmer}.  
A {\em tensor triangulated category} (TTC) is a triple $(\bK, \otimes, \unit)$ such that (i) $\bK$ is a triangulated category, and (ii) 
$\bK$ has a symmetric monodial tensor product $\otimes: \bK \times \bK \rightarrow \bK$ which is exact in each variable with unit object $\unit$.  

If $\bK$ is a TTC one can define notions prevalent in commutative algebra like prime ideal and  spectrum. 
A \emph{(tensor) ideal} in $\bK$ is a triangulated subcategory $\bI$ of $\bK$ such that $M\otimes N\in \bI$ for all $M\in \bI$ and $N\in \bK$. 
Following \cite[Definition 2.1]{balmer}, a {\em prime ideal} $\bP$ of $\bK$ is a proper thick tensor ideal such that if $M\otimes N\in \bP$ then either $M\in \bP$ or $N\in \bP$. 
The {\em Balmer spectrum} \cite[Definition 2.1]{balmer}  is defined as 
$$\operatorname{Spc}(\bK)=\{\bP \subset \bK \mid \bP\ \text{is a prime ideal}\}.$$
The topology on $\operatorname{Spc}({\bK})$ is given by closed sets of the form 
$$Z({\mathcal C})=\{\bP\in \operatorname{Spc}(\bK) \mid {\mathcal C}\cap \bP=\varnothing \}$$ 
where ${\mathcal C}$ is a family of objects in $\bK$. 

When we say $\bK$ is a compactly generated TTC we mean that $\bK$ is closed under arbitrary set indexed coproducts, the tensor product preserves set indexed coproducts, $\bK$ is compactly generated, the tensor product of compact objects is compact, that $\unit$ is a compact object, and that every compact object is rigid (i.e.\  strongly dualizable) as in, for example, \cite{HPS}.   In particular we have an exact contravariant duality functor $(-)^{*}:\bK^{c} \rightarrow \bK^{c}$ such that 
$$\Hom_{{\bK}}(N\otimes M, Q)=\Hom_{\bK}(N,M^{*}\otimes Q)$$ 
for $M\in \bK^{c}$ and $N,Q\in \bK$.

\subsection{Zariski Spaces}\label{SS:zariski}  Assume throughout this paper that $X$ is a Noetherian topological space. In this case any closed set in $X$ is the 
union of finitely many irreducible closed sets. We say that $X$ is a {\em Zariski space} if in addition any irreducible closed set $Y$ of $X$ has a unique generic point (i.e., $y\in Y$ such that $Y=\overline{\{y\}}$). For a 
Zariski space there is a bijective correspondence between points of $X$ and irreducible closed sets of $X$ given by $x\rightarrow \overline{\{x\}}$ for $x\in X$. 
The prototypical example of such a topological space is $X=\operatorname{Spec}(R)$ where $R$ is a commutative Noetherian ring. If $R$ is graded one can also consider $X=\operatorname{Proj}(R)$. 

Let $\XX$ be the collection of all subsets of $X$, $\XX_{cl}$ be the collection of all closed subsets of $X$, and $\XX_{irr}$ be the set of irreducible closed sets. Let $\YY$ be a collection of irreducible closed subsets of $X$. Following \cite[Definition 2.33]{BIK:12}, the collection $\YY$ is {\em specialization closed} if for any $B\in \YY$ and $A\subseteq B$ where $A\in \XX_{irr}$ then $A\in \YY$. Alternatively, a subset $W \subseteq X$ is {\em specialization closed} if $W=\cup_{j\in J} W_{j}$ with $W_{j}\in \XX_{cl}$. One can translate between these definitions and obtain a bijective correspondence in the following way. If $\YY$ is a collection which is specialization closed, one can take $W=\cup_{Y\in \YY}Y$ to obtain a specialization closed set in $X$. On the other hand, if $W=\cup_{j\in J} W_{j}$ is a specialization closed 
subset of $X$, let $\YY$ be the collection $\{ A\in\X_{irr} \mid A\subseteq W_{j} \text{ for some } j\in J\}$. The collection of all specialization closed subsets of $X$ will be denoted by 
$\XX_{sp}$.  

We will be interested in cases where we have an algebraic group $G$ acting rationally on a graded commutative ring $R$ by automorphisms which preserve the grading. This action induces an action of $G$ on $X=\operatorname{Proj}( R)$.  Following \cite{Lorenzpaper}, we can consider $X_{G}=G\operatorname{-Proj}( R)$ which is the set of homogeneous $G$-prime ideals of $R$. There exists a canonical map $\rho:X\twoheadrightarrow X_{G}$ with $\rho( P)=\cap_{g\in G}\ gP=:\cap_{g}\ gP$. The topology on $X_{G}$ is given by declaring $W\subseteq X_{G}$ closed if and only if $\rho^{-1}(W)$ is closed in $X$. An important property to note for us is that $\cap_{g}\ gP_{1}=\cap_{g}\ gP_{2}$ for $P_{1}, P_{2}\in X$ if and only if $\overline{G\cdot P_{1}}=\overline{G\cdot P_{2}}$ in $X$ \cite[(12)]{Lorenzpaper}. 

One can verify that if $W$ is a closed set in $X_{G}$ then $\rho^{-1}(W)$ is a $G$-invariant closed set in $X$. Furthermore, the surjectivity of $\rho$ implies that $\rho(\rho^{-1}(W))=W$. Moreover, if $V$ is a $G$-invariant closed set in $X$ then $V=\rho^{-1}(\rho(V))$ and $\rho(V)$ is closed.  We can conclude that the map $\rho$ induces a one-to-one correspondence between closed sets in $X_{G}$ and $G$-invariant closed sets in $X$. Recall that a subset of $X_{G}$ is \emph{specialization closed} if it is the union of closed sets; equivalently, if its preimage under $\rho$ is the union of $G$-invariant closed sets. Moreover, a closed subset $W$ of $X_{G}$ is irreducible if and only if $\rho^{-1}(W)$ is not the union of proper $G$-invariant closed subsets. 

We claim that $X_{G}$ is a Zariski space.  The fact that $X_{G}$ is Noetherian immediately follows from the fact that $X$ is Noetherian along with the discussion in the previous paragraph. Now let $W$ be an irreducible closed set in $X_{G}$. If $V=\rho^{-1}(W)$ then $V=\overline{G\cdot Y}$ for some irreducible closed set $Y$ because $X$ is Noetherian. Since $X$ is a Zariski space there exists $P\in X$ with $Y=\overline{\{P\}}$, and so $V=\overline{G\cdot P}$. Using the properties of $\rho$ above it follows that 
\[
W=\rho(\overline{G\cdot P})=\overline{\rho(G\cdot P)}=\overline{\{\cap_{g}\ gP\}}=\overline{\{ \rho (P)\}}.
\]
Therefore, $W$ has a generic point. In order to prove uniqueness suppose that $W=\overline{\{\cap_{g}\ gP_{1}\}}=\overline{\{\cap_{g}\ gP_{2}\}}$. It follows that 
$\rho^{-1}(W)=\overline{G\cdot P_{1}}=\overline{G\cdot P_{2}}$. By our earlier remark this implies $\cap_{g}\ gP_{1}=\cap_{g}\ gP_{2}$ and so the generic point is unique.

\subsection{Support data} \label{SS:supportdata} We recall the definition of support data as given in \cite{balmer}, and will view support data as a method to relate objects in a TTC to  subsets in a given Zariski space. Let $\bK$ be a TTC, $X$ be a Zariski space and $\XX$ be the collection of all subsets of $X$. A {\em support data} is an assignment $V:\bK\rightarrow \XX$ which satisfies the following six properties (for $M, M_{i}, N, Q \in \bK$): 
\begin{eqnlist}
\item \label{E:supportone} $V(0)=\varnothing$, $V(\unit)=X$;
\item \label{E:supporttwo} $V( \oplus_{i \in I} M_{i})=\bigcup_{i \in I} V(M_{i})$ whenever $ \oplus_{i \in I} M_{i}$ is an object of $\bK$;
\item \label{E:supportthree} $V(\Sigma M)=V(M)$;
\item \label{E:supportfour} for any distinguished triangle $M\rightarrow N \rightarrow Q \rightarrow \Sigma M$ we have 
$$V(N)\subseteq V(M)\cup V(Q);$$ 
\item \label{E:supportfive} $V(M\otimes N)=V(M)\cap V(N)$;
\end{eqnlist} 
Using the above properties and Theorem A.2.5 and Lemma A.2.6 of \cite{HPS} it is straightforward to verify that
\begin{eqnlist} 
\item \label{E:supportsix} $V(M)=V(M^{*})$ for $M\in\bK^{c}$.
\end{eqnlist} 
We will be interested in support data which satisfy an additional two properties: 
\begin{eqnlist} 
\item \label{E:supportseven} $V(M)=\varnothing$ if and only if $M=0$;
\item \label{E:supporteight} for any $W\in \XX_{cl}$ there exists  $M\in \bK^{c}$ such that $V(M)=W$ (Realization Property). 
\end{eqnlist}

We note the following lemma which says roughly that ``passing to a localizing subcategory does not increase supports.''

\begin{lemma} \label{L:localizingsubcategorysupport}
Let $\bK$ be a TTC which is closed under set indexed coproducts and admits a support data $V:\bK \to \XX$.  Let $\CC$ be a collection of objects in $\bK$ and suppose $W$ is a subset of $X$ such that $V(M)\subseteq W$ for all $M\in \mathcal C$. Then $V(M)\subseteq W$ for all $M\in\Loc(\mathcal C)$. 
%
%Likewise if $\V$ satisfies (\ref{SS:supportdata}.\ref{E:supportfive}), and if $\TT$ is the tensor ideal of $\bK$ or $\bK^{c}$ generated by $\CC$, then $\V(M) \subseteq W$ for all $M\in\TT$.
\end{lemma}

\begin{proof}
In forming $\Loc(\mathcal C)$, we iteratively adjoin new objects to the collection by applying the following constructions:
\begin{itemize}
\item apply $\Sigma$ and $\Sigma^{-1}$ to objects in our collection
\item if two objects in a distinguished triangle are in our collection, add the third object to our collection
\item take direct summands of objects in our collection
\item take set indexed coproducts of objects in our collection.
\end{itemize}   The properties of a support data imply that the support of any such new object is still contained in $W$.
%
%In forming $\TT$, we additionally tensor by objects in $\bK$ or $\bK^{c}$. But according to (\ref{SS:supportdata}.\ref{E:supportfive}), this operation also reduces supports. The claim about $\TT$ follows.
\end{proof}

It is useful for support computations (see \cref{S:VarietiesandGeometricInduction}) to also note that when working with closed sets in the spectrum we can instead work with the maximal ideal spectrum of $R$.    
Namely, let $R$ be a finitely generated, commutative, graded $k$-algebra with $k$ algebraically closed and $X=\operatorname{Proj}(R)$, and $X_{\max}=
\operatorname{Proj}(\operatorname{MaxSpec}(R))$. For any ideal $I$ in $R$, set $v(I)=\{P\in X \mid P\supseteq I\}$ be the closed set in $X$ defined by $I$ and let $Z(I)$ be the zero locus of $I$ in $X_{\max}$.  Let $W$ be a closed set in $X$ and consider $W_{\max}=W\cap X_{\max}$ which is a closed 
set in $X_{\max}$. Let $V: \bK \to \XX$ be a support data as in \cref{SS:supportdata}.  We claim that if there exists $M$ in $\bK$ such that $V(M)$ is closed and $V_{\max}(M):=V(M)\cap X_{\max}=W_{\max}$ then $V(M)=W$. 
Let $I$ and $J$ be ideals of $R$ such that $v(I)=W$ and $v(J)=V(M)$. Since $V_{\max}(M)=W_{\max}$ it follows that $Z(I)=Z(J)$ and, hence, the radicals of $I$ and $J$ coincide.  This then implies $v(I)=v(J)$, as claimed.  Therefore, when $X$ is the spectrum of a finitely generated commutative $k$-algebra we may verify (\ref{SS:supportdata}.\ref{E:supporteight}) for $X$ by instead verifying (\ref{SS:supportdata}.\ref{E:supporteight}) for $X_{\max}$.  In the case when a group $G$ acts on $R$  we can apply the same reasoning to see that to verify (\ref{SS:supportdata}.\ref{E:supporteight}) for $X_{G}$ it suffices to verify it for $X_{G, \max}$.  In particular, using $\rho$ we see that it suffices to realize the $G$-invariant closed subsets of $X_{\max}$.

\subsection{Extending a Support Data} \label{SS:extending} 

We also have the notion of extending a support data.  Namely, let $X$ be a Zariski space, and recall the notation $\XX$ (resp.\ $\XX_{cl}$) for the collection of all (resp.\ all closed) subsets of $X$.  Assume $\bK$ is a TTC and $V: \bK^{c} \to \XX_{cl}$ is a support data.  
\begin{defn}\label{D:extends}
We say that $\VV:\bK\rightarrow \XX$  \emph{extends} $V: \bK^{c} \to \XX_{cl}$ if 
\begin{itemize} 
\item[(i)] $\VV$ satisfies properties (\ref{SS:supportdata}.\ref{E:supportone})--(\ref{SS:supportdata}.\ref{E:supportfive}) for objects in $\bK$; 
\item[(ii)] $\VV(M)=V(M)$ for all $M\in \bK^{c}$; and
\item[(iii)]  if $V$ satisfies (\ref{SS:supportdata}.\ref{E:supportseven}) then $\VV$ satisfies (\ref{SS:supportdata}.\ref{E:supportseven}).
\end{itemize} 
\end{defn}

\section{Localization and the Classification Theorems}\label{S:classificationtheorems}

\subsection{Localization functors}\label{SS:localization} In this section we identify a key tool: the localization and colocalization 
functors as given in \cite[Section 3]{BIK:08a}.  Bousfield localization has its origins in homotopy theory, but starting with the work of Neeman and Rickard it has proven invaluable to the study of triangulated categories in representation theory and other settings.  For details about these localization functors, we refer the reader to \cite[Section 3]{BIK:08a}. The following theorem is a restatement of \cite[Theorem 2.32]{BIK:12} in our setting and, while they work only in the case of finite groups, the same proof applies here.  It was originally proven for finite groups by Rickard \cite{Rick}.  Alternatively it follows from \cite[Proposition 2.9]{Dell:10} by taking $\alpha$ to be the cardinality of a proper class.

\begin{theorem} \label{T:localizationtriangles} Let $\bK$ be a compactly generated triangulated category.  Given a thick subcategory $\bC$ of $\bK^{c}$ and an object 
$M$ in $\bK$, there exists a functorial triangle in $\bK$,
$$
\Gamma_{\bC}(M) \to M \to L_{\bC}(M) \to
$$
which is unique up to isomorphism, such that $\Gamma_{\bC}(M)$ is in $\Loc(\bC)$ and there are no non-zero maps in $\bK$ from $\bC$ or, equivalently, from $\Loc(\bC)$ to $L_{\bC}(M)$.
\end{theorem}

\iffalse
\begin{proof} For the sake of completeness, we outline the proof in \cite[Theorem 2.32]{BIK:12}, adding a few details. The subcategory $\Loc(\bC)$ of $\bK$ is compactly generated, since $\bC$ consists of compact objects. The inclusion $F:\Loc(\bC) \to \bK$ certainly preserves coproducts, so {\it op.\ cit.\ }Corollary 2.13 to Brown representability yields a right adjoint $G:\bK \to \Loc(\bC)$. Then, since $F$ is also fully faithful, the colocalization analogue of {\it op.\ cit.\ }Lemma 2.14 shows that $\Gamma_{\bC}:= FG:\bK\to\bK$ is a colocalization functor.
Finally, the existence of a corresponding localization functor $L_{\bC}$ and the other claims follows from {\it op.\ cit.}\ Proposition 2.16.
\end{proof}
\fi

We also record the following fact which will be needed in the proof of \cref{T:Hopkinstheorem}.

\begin{lemma} \label{L:gammaC}
Let $\bK$ be a compactly generated triangulated category, and $\bC$  a thick subcategory of $\bK^{c}$. Given an object $M\in\bK $,  
$$
M \in \Loc(\bC) \quad \text{if and only if} \quad \Gamma_{\bC}(M)\cong M.
$$
\end{lemma}

\begin{proof}
Suppose that $M \in \Loc(\bC)$. By \cref{T:localizationtriangles}, there are no non-zero maps from $\Loc(\bC)$ to $L_{\bC}(M)$. In particular, the last map in the triangle
$$\Sigma^{-1} L_C(M) \to  \Gamma_C(M) \to  M  \xrightarrow{0} L_C(M) \to $$
must be zero. Now one can apply \cite[Corollary 1.2.7]{Nee:01} to the triangle
to deduce that $\Gamma_C(M) \cong  M \oplus \Sigma^{-1} L_C(M)$. Since there are no nonzero
maps from $\Gamma_C(M)$ to $\Sigma^{-1} L_C(M)$ it follows that $L_C(M) = 0$, and hence $\Gamma_{\bC}(M)\cong M$.

Conversely, suppose that $\Gamma_{\bC}(M)\cong M$.  In the proof of \cite[Theorem 2.32]{BIK:12}, $\Loc(\bC)$ (resp.\ $\Gamma_{\bC}$) is playing the role of $\bS$ (resp.\ $\Gamma$) in \cite[Proposition 2.16]{BIK:12}. But that result asserts that $\Im(\Gamma) = \bS$. So we can conclude that $\Im(\Gamma_{\bC}) = \Loc(\bC)$. However, 
$$\Im(\Gamma_{\bC}) := \{ N \in \bK  \mid N \cong \Gamma_{\bC}(Q) \text{ for some } Q \in \bK  \}.$$ 
Since $M \cong \Gamma_{\bC}(M)$, it follows that $M \in \Im(\Gamma_{\bC}) = \Loc(\bC)$.
\end{proof}

\subsection{Extending Support Data via Localization Functors} \label{SS:extend}  Let $X$ be a Zariski space, $\XX_{cl}$ its closed subsets, and $\XX_{sp}$ its specialization closed subsets.   
As a first application of localization and colocalization functors, given a support data on $\bK^{c}$, $V:\bK^{c}\rightarrow \XX_{cl}$, we can construct an 
assignment on $\bK$, $\VV : \bK \to \XX$, which is near to being an extension of $V$.  

\begin{defn}\label{D:abstractsupport} Let $\bK$ be a compactly generated TTC and let $V:\bK^{c}\rightarrow \XX_{cl}$ be a support data. 
\begin{itemize}
\item[(a)] Given $W\in \XX$, let $\bI_{W}$ denote the thick tensor ideal of $\bK^{c}$ given by all $M\in \bK^{c}$ such that $V(M) \subseteq W$. 

\item [(b)] Using \cref{T:localizationtriangles} set 
$$
\Gamma_{W} = \Gamma_{\bI_{W}} \quad \text{and} \quad L_{W}=L_{\bI_{W}}.
$$
Given $\varnothing\ne W\in \XX_{irr}$, let $Z=\{x \in X \mid W \not\subseteq\overline{\{x\}} \}$. Define
$$
\nabla_{W}=\Gamma_{W} L_{Z} = L_{Z} \Gamma_{W};
$$
cf.\ \cite[Proposition 6.1]{BIK:08a}.
\end{itemize} 
\end{defn}

\begin{defn}\label{D:abstractsupport2} Let $\bK$ be a compactly generated TTC, and $V:\bK^{c}\rightarrow \XX_{cl}$ be a support data. Define $\VV: \bK \to \XX$ as follows.   For $M\in\bK$ set
$$
\VV(M) =  \{x\in X \mid \nabla_{\overline{\{x\}}}(M) \ne 0 \};
$$ cf.\ \cite[Definition 2.35]{BIK:12}.
\end{defn}

The following theorem demonstrates that $\VV:\bK \rightarrow \XX$ has properties  (\ref{SS:supportdata}.\ref{E:supportone})--(\ref{SS:supportdata}.\ref{E:supportfour})  for a support data whenever (\ref{SS:supportdata}.\ref{E:supporteight}) holds for $V$. 

\begin{theorem} \label{T:AbstractSupportProperties}  Let $M,\ N,\ Q,\ M_{\a}\in\bK$ for $\a$ in some index set.
\begin{enumerate}
\item[(i)] $\VV(0) = \varnothing$.
\item[(ii)]  $\VV(M\oplus N) = \VV(M) \cup \VV(N)$, and more generally $\VV(\oplus_{\a}M_{\a}) = \bigcup_{\a}\VV(M_{\a})$.
\item[(iii)] $\VV(\Sigma  M)=\VV(M)$.
\item[(iv)] for any distinguished triangle $M\rightarrow N \rightarrow Q \rightarrow \Sigma M$, 
$$\VV(N)\subseteq \VV(M)\cup \VV(Q).$$ 
\end{enumerate}
Finally, assuming that (\ref{SS:supportdata}.\ref{E:supporteight}) holds for $V$, then 
\begin{enumerate}
\item[(v)] $\V(\unit)=X$.
\end{enumerate} 
\end{theorem}

\begin{proof} (i) If $M=0$ then $\nabla_{W}(M)=0$ for all $W\in \X_{irr}$, implying that $\V(M)=\varnothing$. (ii) $\nabla_{W}(M\oplus N) = \nabla_{W}(M)\oplus \nabla_{W}(N)$ since $
\nabla_{W}$ is an exact functor and, in particular, is additive (cf.\ \cite[Section 3]{BIK:08a}). Thus $\nabla_{W}(M\oplus N) \ne 0$ if and only if $\nabla_{W}(M)\ne 0$ or $\nabla_{W}(N)\ne 0$. The same argument works for arbitrary direct sums, and the result follows. 
(iii) Being exact, $\nabla_{W}$ commutes with $\Sigma$, which implies the result (cf.\ \cite[Proposition 5.1]{BIK:08a} and \cite[Section 1.3.6]{BIK:12}). (iv) By exactness of $\nabla_{W}$, $\nabla_{W}(M) \to \nabla_{W}(N) \to \nabla_{W}(Q) \to \nabla_{W}(\Sigma M)$ 
is a distinguished triangle. If $\nabla_{W}(N)\ne 0$, then at least one of $\nabla_{W}(M)\ne 0$ or $\nabla_{W}(Q)\ne 0$.

Let $x\in X$ and $W=\overline{\{x\}}$. Assuming that (\ref{SS:supportdata}.\ref{E:supporteight}) holds for $V$, for $W$ there exists $M\in\bK^{c}$ with $V(M)=W$. Since $M\in\Loc(\bI_{W})$, we have $\Gamma_{W}(M) \cong M$ 
(cf.\ Definition~\ref{D:abstractsupport} and Lemma~\ref{L:gammaC}). Letting $Z$ be as in Definition~\ref{D:abstractsupport}, we have $M\notin\bI_{Z}$. 
It follows from \cref{L:gammaC} again that $\Gamma_{\bI_{Z}}(M) \not\cong M$ and hence (from the exact triangle) that $L_{Z}(M)\ne 0$. Thus $\nabla_{W}(M) = L_{Z} \Gamma_{W}(M) \cong L_{Z}(M) \ne 0$. But $\nabla_{W}(M) 
\cong \nabla_{W}(\unit \otimes M) \cong \nabla_{W}(\unit ) \otimes M$ by \cite[Corollary 8.3]{BIK:08a}. Thus $\nabla_{W}(\unit )\ne 0$ and $x\in \V(\unit)$. This proves that $\V(\unit )=X$.
\end{proof}

Thanks to \cref{T:AbstractSupportProperties}, in order to show that $\VV$ extends $V$ it suffices to prove that 
\begin{eqnlist} 
\item \label{E:extendone1} $\VV(M \otimes N) = \VV(M) \cap \VV(N)$ for $M, N\in \bK$; 
\item \label{E:extendtwo2}  $\VV(M)=V(M)$ for all $M \in \bK^{c}$; and,
\item \label{E:extendthree3}  if $V$ satisfies (\ref{SS:supportdata}.\ref{E:supportseven}), then $\VV$ satisfies (\ref{SS:supportdata}.\ref{E:supportseven}).
\end{eqnlist}

\subsection{Hopkins' Theorem}  Given an object $M\in\bK^{c}$, let $\Tensor(M)\subseteq \bK^{c}$ be the thick tensor ideal in $\bK^{c}$ generated by $M$. 
We can also use localization to provide a general version of a theorem of Hopkins \cite{Hop} and Neeman \cite{Nee2} in the context of our setting.  Note that 
the following theorem holds in general for an arbitrary extension ${\mathcal V}$ of a support data $V$ (cf.\ Definition~\ref{D:extends}). 

\begin{theorem} \label{T:Hopkinstheorem} Let $\bK$ be a compactly generated TTC, $X$ be a Zariski space, and $\XX_{cl}$ be the closed subsets of $X$. Moreover, let $V:\bK^{c}\rightarrow \XX_{cl}$ be a 
support data satisfying (\ref{SS:supportdata}.\ref{E:supportseven}), and assume that $\VV: \bK \to \XX$ extends $V$. Fix an object $M\in\bK^{c}$, and set $W= V(M)$. 
Then $\bI_{W}=\Tensor(M)$, where $\bI_{W}$ is as in \cref{D:abstractsupport}.
\end{theorem}

\begin{proof} We apply the strategy presented in \cite[Theorem 2.39]{BIK:12} which in turn is based upon the analogous result in \cite{BCR}. For brevity we set $\bI=\bI_{W}$ and $ \bI'=\Tensor(M)$. 

\medskip
\noindent ($\supseteq$) Since $\Tensor(M)$ is the smallest thick tensor ideal of $\bK^{c}$ containing $M$, from the definition of $\bI_{W}$ it follows that $\bI \supseteq \bI'$.

\medskip
\noindent ($\subseteq$) Let $N\in\bK$. Apply the exact triangle of functors $\Gamma_{\bI'}\to\text{Id} \to L_{\bI'}\to\ $ to $\Gamma_{\bI}(N)$:
\begin{equation} \label{E:ExactTriangle}
\Gamma_{\bI'}\Gamma_{\bI}(N) \to \Gamma_{\bI}(N) \to L_{\bI'} \Gamma_{\bI}(N) \to
\end{equation}
Since $\bI' \subseteq \bI$, the first term belongs to $\Loc(\bI') \subseteq \Loc(\bI)$. The second term also belongs to $\Loc(\bI)$, a triangulated subcategory, and hence so does the third term: $L_{\bI'}\Gamma_{\bI}(N) \in \Loc(\bI)$. Using  \cref{L:localizingsubcategorysupport}, we deduce that $\VV(L_{\bI'}\Gamma_{\bI}(N)) \subseteq W$.

By \cref{T:localizationtriangles} (applied to $\bI'$), there are no non-zero maps from $\bI'$ to $L_{\bI'}\Gamma_{\bI}(N)$. Thus, for any object $S\in\bK^{c}$, the duality property implies 
that 
\begin{equation} \label{E:ExactNTriangle}
0=\Hom_{\bK}(S\otimes M, L_{\bI'}\Gamma_{\bI}(N)) \cong \Hom_{\bK}(S, M^{*}\otimes L_{\bI'}\Gamma_{\bI}(N)).
\end{equation}
Since $\bK$ is compactly generated it follows that 
$M^{*}\otimes L_{\bI'}\Gamma_{\bI}(N)=0$ in $\bK$. Hence, 
\begin{align*}
\varnothing &= \VV(M^{*}\otimes L_{\bI'}\Gamma_{\bI}(N)) \\
  &= \VV(M) \cap \VV(L_{\bI'}\Gamma_{\bI}(N)) \\
  &= W \cap \VV(L_{\bI'}\Gamma_{\bI}(N)) \\
  &= \VV(L_{\bI'}\Gamma_{\bI}(N)),
\end{align*}
by Definition~\ref{D:extends}, (\ref{SS:supportdata}.\ref{E:supportsix}), and the observation at the end of the previous paragraph. Thus, by Definition~\ref{D:extends}(iii), $L_{\bI'}\Gamma_{\bI}(N) =0$ in $\bK$. 
By \cref{E:ExactTriangle} it follows that $\Gamma_{\bI}(N) \cong \Gamma_{\bI'}\Gamma_{\bI}(N)$.

Now specialize to $N\in\bI$. Then, using \cref{L:gammaC} twice, we have $\Gamma_{\bI}(N) \cong N$, so $\Gamma_{\bI'}(N) \cong N$, whence $N\in\operatorname{Loc}\left( \bI'\right)$. Applying \cite[Lemma 2.2]{Nee} we see that in fact $N \in \bI'$.  This shows $\bI\subseteq\bI'$ and completes the proof.
\end{proof}

\subsection{Classifying Thick Tensor Ideals in a TTC} \label{SS:classifying} The following theorem provides a classification of the thick tensor ideals of compact objects in our setting.  Dell'Ambrogio proved a somewhat more general version of this result \cite[Theorem 1.5]{Dell:10}.

\begin{theorem} \label{I:bijectiongeneral} Let $\bK$ be a compactly generated TTC.  Let $X$ be a Zariski space and let
$V:\bK^{c} \rightarrow \XX_{cl}$ be a support data defined on $\bK^{c}$ satisfying the additional conditions (\ref{SS:supportdata}.\ref{E:supportseven}) and (\ref{SS:supportdata}.\ref{E:supporteight}).  Moreover, assume $\VV: \bK \to \XX$ extends $V$. 

Given the above setup there is a pair of mutually inverse maps
$$
\{\text{thick tensor ideals of $\bK^{c}$}\} \begin{array}{c} {\Gamma} \atop {\longrightarrow} \\ {\longleftarrow}\atop{\Theta} \end{array}  \XX_{sp},
$$
given by 
\begin{align*}
\Gamma(\bI) & =\bigcup_{M\in \bI} V(M),\\
 \Theta(W) & = \bI_{W},
\end{align*}
where $\bI_{W}=\{M\in \bK^{c} \mid V(M)\subseteq W \}$ as in \cref{D:abstractsupport}. 
\end{theorem}

\begin{proof} Observe that by (\ref{SS:supportdata}.\ref{E:supporttwo})--(\ref{SS:supportdata}.\ref{E:supportfive}), $\bI_{W}$ is a thick tensor ideal of $\bK^{c}$.  

(i) We first show that $\Gamma \circ \Theta$ is the identity. Observe that 
$$
\Gamma(\Theta(W)) = \Gamma(\bI_{W}) = \bigcup_{M\in\bI_{W}} V(M) \subseteq W,
$$
where the final inclusion follows from definition of $\bI_{W}$.

For the reverse inclusion, write $W=\bigcup_{j\in J} W_{j}$ for some index set $J$ and closed subsets $W_{j}\in \XX$. By (\ref{SS:supportdata}.\ref{E:supporteight}), there exist objects $N_{j}\in \bK^{c}$ such that $V(N_{j}) = W_{j}$ for 
$j\in J$. Then $N_{j}\in \bI_{W}$ so $W \subseteq \bigcup_{M\in\bI_{W}} V(M)$. Hence, $\Gamma(\Theta(W))=W$.

(ii) Now we show that $\Theta \circ \Gamma$ is the identity. Given a thick tensor ideal $\bI$, set $W=\Gamma(\bI) = \bigcup_{M\in\bI} V(M)$. Then
$$
\Theta(\Gamma(\bI)) = \Theta(W)=\bI_{W}\supseteq \bI.
$$

For the reverse inclusion, let $N\in\bI_{W}$, so $V(N) \subseteq W$. Since $X$ is Noetherian, $V(N)=W_{1}\cup\dots\cup W_{n}$, where the $W_{i}$ are the irreducible components of $V(N)$. 
Each $W_{i}$ has a generic point $x_{i}$ with $\overline{\{x_{i}\}}=W_{i}$, and since $W_{i}\subseteq W$, $x_{i}\in W$. By definition of $W$, there  exists $M_{i}\in\bI$ such that $x_{i}\in V(M_{i})$. Since $V(M_{i})$ 
is closed, $W_{i}\subseteq V(M_{i})$. Set $M:= \bigoplus_{i=1}^{n}M_{i}\in\bI$. Then
$$
V(N) \subseteq \bigcup_{i=1}^{n}V(M_{i}) = V(M) \subseteq W.
$$

We claim that $N\in\Tensor(M)$. Since $\bI$ is a thick tensor ideal containing $M$, clearly $\Tensor(M)\subseteq\bI$, so this assertion will complete the proof of the  inclusion $\bI_{W}\subseteq\bI$, and thus of the 
theorem.

To prove the claim, we use Theorem~\ref{T:Hopkinstheorem}. Namely, we have the equality
$\Tensor(M) = \bI_{Z}$, where 
$$
Z =V(M).
$$
But $V(N)\subseteq Z$ since $V(N)\subseteq V(M)$, so $N\in\bI_{Z}=\Tensor(M)$ as claimed.
\end{proof}

\subsection{Computing the Balmer Spectrum} \label{SS:balmer}

We can now prove that the Balmer spectrum of $\bK^{c}$ is homeomorphic to $X$ in our setup.  %For the reader who refers to \cite{balmer} we point out a difference in notation between here and here.  Namely a support data on $\bK^{c}$ gives as output a closed set in a topological space, $X$, whereas here we view it as function to $\XX$, the set of all closed subsets of $X$. 

\begin{theorem} \label{K:bijectiongeneral} Let $\bK$ be a compactly generated TTC and let $X$ be a Zariski space. Assume that  
$V:\bK^{c} \rightarrow \XX_{cl}$ is a support data defined on $\bK^{c}$ satisfying the additional conditions (\ref{SS:supportdata}.\ref{E:supportseven}) and (\ref{SS:supportdata}.\ref{E:supporteight}).  Further assume that we have a support data $\VV: \bK \to \XX$ which extends $V$. Then there is a homeomorphism
$$
f: X  \to \operatorname{Spc}(\bK^{c}).
$$
\end{theorem}

\begin{proof}  Since $V:\bK^{c} \rightarrow \XX_{cl}$ is a support data the universal property of the Balmer spectrum implies that there is a continuous map $f: X  \to \operatorname{Spc}(\bK^{c})$ as defined in \cite[Theorem 3.2]{balmer}.  By assumption $X$ is Zariski and by \cref{I:bijectiongeneral} we know that $V$ induces a bijection between the thick tensor ideals of $\bK^{c}$ and the specialization closed subsets of $X$ (noting that since $\bK^{c}$ is assumed to have a duality we have that the radical thick tensor ideals are exactly the thick tensor ideals by \cite[Remark 4.3 and Proposition 4.4]{balmer}).  Thus the pair $(X,V)$ is a classifying support data as defined in \cite[Definition 5.1]{balmer}.  The results in \cite[Theorem 5.2]{balmer} then implies that $f$ is a homeomorphism.
\end{proof}

\section{Classical Lie Superalgebras} \label{S:classicalLiesuperalgebras} 

\subsection{}\label{SS:notationandconventions} We now apply the technology of the previous sections to the representations of classical Lie superalgebras.  We start with 
the notation and conventions in \cite{BKN1, BKN2, BKN3, LNZ}. Let $\fa=\fa_{\0}\oplus \fa_{\1}$ be a 
Lie superalgebra over the complex numbers, $\C$, with Lie bracket $[\;,\;]:\fa \otimes \fa \to \fa$. A finite dimensional Lie superalgebra
$\fa$ is called \emph{classical} if there is a connected reductive algebraic group $A_{\0}$ such that $\operatorname{Lie}(A_{\0})=\fa_{\0}$ and an action of $A_{\0}$ on $\fa_{\1}$ which differentiates to the adjoint action of $\fa_{\0}$ on $\fa_{\1}.$ In what follows all Lie superalgebras will be assumed to be classical.

Let $U(\fa)$ be the universal enveloping superalgebra of $\fa$. If $M$ and $N$ are $\fa$-supermodules (equivalently 
$U(\fa)$-supermodules) one can use the coproduct and antipode of $U(\fa)$ to define an $\fa$-supermodule
structure on the tensor product $M\otimes N$ and, when $M$ is finite dimensional, the dual $M^{*}$. In this context we will use the term $\fa$-module to mean 
$\fa$-supermodule. Furthermore, in what follows we will only consider the underlying even category in which the Hom sets consist of morphisms of $\fa$-modules 
which preserve the $\Z_{2}$-grading. 

Let $\FF_{(\fa,\fa_{\0})}$ be the full subcategory of
finite dimensional $\fa$-modules which have a compatible action by $A_{\0}$ and are completely reducible over $A_{\0}$.  The category $\FF:=\FF_{(\fa,\fa_{\0})}$ has enough injective
(and projective) modules and is a Frobenius category (cf.\ \cite[Proposition 2.2.2]{BKN3}). Given $M, N$ in $\FF$, let $\Ext_{\mathcal{F}}^{d}(M,N)$ be the degree $d$
extensions between $N$ and $M$. These extension groups can be realized via relative Lie superalgebra cohomology for the pair
$(\fa,\fa_{\0})$: 
$$\Ext_{\mathcal{F}}^{d}(M,N)\cong \Ext^{d}_{(\fa,\fa_{\0})}(M,N)\cong \HH^{d}(\fa,\fa_{\0}; M^{*}\otimes N).$$
There exists a Koszul type resolution which can be used to calculate $\HH^{d}(\fa,\fa_{\0}; M^{*}\otimes N)$. 

\subsection{} Let  $\CC_{(\fa,\fa_{\0})}$ be the category of all $\fa$-modules which are finitely semisimple over $A_{\0}$ 
(i.e., the objects in $\CC_{(\fa,\fa_{\0})}$ have a compatible $A_{\0}$ action and as $A_{\0}$-modules they decompose into a direct sum of finite dimensional simple modules).  
The category $\FF_{(\fa,\fa_{\0})}$ is a full subcategory of $\CC_{(\fa,\fa_{\0})}$. Since the projectives, injectives, and simple modules 
in $\CC_{(\fa,\fa_{\0})}$ are the same as in $\F_{(\fa,\fa_{\0})}$, it follows that $\CC_{(\fa,\fa_{\0})}$ is also a Frobenius category. 

Let $\bK=\operatorname{Stab}(\CC_{(\fa,\fa_{\0})})$ be the stable module category of $\CC_{(\fa,\fa_{\0})}$.  That is, it is the quotient category of  $\CC_{(\fa,\fa_{\0})}$ given by identifying maps whose difference factors through a projective module.  Since the projective and injective modules in $\CC_{(\fa,\fa_{\0})}$ coincide it follows that $\bK$ is a triangulated category. See \cite[Section 1.3.5]{BIK:12} for further details. Using the coproduct and augmentation maps on $U(\fg)$ it follows that $\bK$ is a tensor triangulated category. We also note 
that $\operatorname{Stab}(\F_{(\fa,\fa_{\0})})$ is a tensor triangulated subcategory of $\bK$. We now verify that $\bK$ is compactly generated 
and $\bK^{c}=\operatorname{Stab}(\F_{(\fa,\fa_{\0})})$. 

\begin{prop} Let $\bK = \operatorname{Stab}(\CC_{(\fa,\fa_{\0})})$. 
\begin{itemize} 
\item[(a)] $\bK$ is generated by $\operatorname{Stab}(\F_{(\fa,\fa_{\0})})$.
\item[(b)] $\bK^{c}= \operatorname{Stab}(\F_{(\fa,\fa_{\0})})$. 
\end{itemize} 
In particular, $\bK$ is compactly generated.
\end{prop} 
\begin{proof} We first observe that since every object in $\FF_{(\fa,\fa_{\0})}$ is finite dimensional they give compact objects in $\bK$.

We now prove $\bK$ is compactly generated.  Let $M \in \bK$ be non-zero.  As in \cite[Lemma 3.1]{Rick} we may assume $M$ is stably isomorphic to an $\fa$-module,  which we also call $M$, with no projective direct summands.  By the PBW theorem for Lie superalgebras the module $M$ is locally finite.  Because of this the $U(\fa)$-module generated by a fixed non-zero element of $M$ is finite dimensional and, hence, is an object of $\F_{(\fa,\fa_{\0})}$.  Call it $N$.  Furthermore, $N$ cannot have a projective direct summand for, if it did, it would provide a projective submodule of $M$ and, hence (since projectives are injective), a projective summand of $M$.   The inclusion map defines an isomorphism between $N$ and its image in $M$ in 
$\CC_{(\fa,\fa_{\0})}$.  But since $M$ and $N$ have no projective summands \cite[Lemma 3.1(b)]{Rick} implies this map also defines a stable isomorphism between $N$ and its image.  In particular, it defines a non-zero homomorphism between $N$ and $M$ in $\bK$.  Therefore $\bK$ is generated by  $\operatorname{Stab}(\F_{(\fa,\fa_{\0})})$ and, in particular, $\bK$ is compactly generated. To prove (b) we note that since 
$\bK = \Loc_{\bK}\left(\operatorname{Stab}(\FF_{(\fa,\fa_{\0})}\right)$ it follows from \cite[Lemma 2.2]{Nee} that $\bK^{c}$ 
is precisely the thick subcategory generated by $\operatorname{Stab}(\FF_{(\fa,\fa_{\0})})$.   That is, $\bK^{c}$ is $\operatorname{Stab}(\FF_{(\fa,\fa_{\0})})$ itself.
\end{proof}

The antipode map on $U(\fa)$ gives a duality on $\FF_{(\fa, \fa_{\0})}$ and from this it follows that the compact objects of $\bK$ are rigid. It is straightforward to verify that $\bK$ is a compactly generated TTC.

\subsection{Relative Cohomology and the Cohomological Support}\label{SS:relativecohomology} The relative cohomology ring for $(\fa,\fa_{\0})$ can be computed using an explicit complex (cf.\ \cite[Section 2.3]{BKN1}). Since $[\fa_{\1},\fa_{\1}]\subseteq \fa_{\0}$ the differentials in the complex are zero and from this it follows that 
$$
R:=\HH^{\bullet}(\fa,\fa_{\0}; \C)=
S^{\bullet}(\fa_{\1}^*)^{A_{\0}}.
$$
Since $A_{\0}$ is assumed to be reductive it follows that $R$ is a finitely generated commutative $\C$-algebra. Moreover, it was shown in \cite{BKN1} that 
if $M_{1}$, $M_{2}$ are in $\FF_{(\fa,\fa_{\0})}$ then $\Ext_{(\fa,\fa_{\0})}^{\bullet}(M_{1},M_{2})$ is a finitely generated $R$-module.

If $X=\operatorname{Proj}(R)$, then $X$ is a Zariski space. Recall that $\bK^{c}=\operatorname{Stab}(\F_{(\fa,\fa_{\0})})$, and that $\XX$ denotes the collection of closed subsets of $X$. 
We will now define an assignment $V_{(\fa,\fa_{\0})}:\bK^{c}\rightarrow \XX_{cl}$ as follows. Let 
$$V_{(\fa,\fa_{\0})}(M):=\operatorname{Proj}(R/J_{\fa, M})$$ 
where $J_{\fa, M}=\operatorname{Ann}_{R}(\Ext_{\mathcal{F}}^{\bullet}(M,M))$
(i.e., the annihilator ideal of this module).  An equivalent formulation of this assignment is 
$$V_{(\fa,\fa_{\0})}(M)=\operatorname{Proj}(\{P\in \operatorname{Spec}(R)\mid \Ext_{(\fa,\fa_{\0})}^{\bullet}(M,M)_{P}\neq 0\})$$   
(i.e., the primes ideals $P$ for which the cohomology is non-zero when localized at $P$).

More generally, for a pair of modules $M,N \in {\mathcal F}_{(\fa,\fa_{\0})}$, let 
$$V_{(\fa,\fa_{\0})}(M,N)=\operatorname{Proj}(\{P\in \operatorname{Spec}(R)\mid \Ext_{(\fa,\fa_{\0})}^{\bullet}(M,N)_{P}\neq 0\}).$$
Note that $V_{(\fa,\fa_{\0})}(M)=V_{(\fa,\fa_{\0})}(M,M)$. 
   
% Note that $J_{\fa, M}=\text{Ann}_{R}\left(  \text{Id}\right)$ where $\text{Id}$ is the identity morphism in $\Ext^{0}_\FF(M,M)$.

As in the case with group algebras, one can verify that properties (\ref{SS:supportdata}.\ref{E:supportone})--(\ref{SS:supportdata}.\ref{E:supportfour}), (\ref{SS:supportdata}.\ref{E:supportsix}) hold for $V_{(\fa,\fa_{\0})}$ (cf.\ \cite[Chapters 8, 10]{Evensbook}). We call $V_{(\fa,\fa_{\0})}$ a \emph{pre-support data} on $\bK^c$ as we will need additional conditions on $\fa$ to have an actual support data which satisfies property  (\ref{SS:supportdata}.\ref{E:supportfive}). 

\subsection{Detecting Subalgebras}\label{SS:detectingsubalgebras} We say that a Lie superalgebra is of {\em Type I} if it admits a $\Z$-grading $\fg=\g_{-1}\oplus {\g}_{0}\oplus {\g}_{1}$  which is compatible with the Lie bracket and with $\fg_{\0}=\fg_{0}$ and
$\fg_{\1}=\g_{-1}\oplus {\g}_{1}$. Otherwise, $\fg$ is of {\em Type II}. 
The prototypical example of a Type I classical Lie superalgebra is $\mathfrak{gl}(m|n)$.  Simple Type I classical Lie superalgebras in the Kac classification \cite{Kac1} include those of types $A(m,n)$, $C(n)$ and $P(n)$.

Throughout this section we will assume that $\fg$ is a Type I classical Lie superalgebra admitting both a stable and 
polar action of $G_{\0}$ on $\fg_{\1}$.  See \cite[Section 3]{BKN1} for the definition of stable and polar, and see Table~5 in \emph{loc.\  cit.} for a list of classical Lie superalgebras which are stable and polar.  Note that this is the same setup used in \cite{LNZ}. We quickly 
review the construction of the two families of  detecting subalgebras for 
classical Lie superalgebras using the invariant theory of $G_{\0}$ on $\fg_{\1}$. For further details we refer the reader to 
\cite[Sections 3--4]{BKN1}. 

If the action of $G_{\0}$ on $\fg_{\1}$ is a stable action, then we 
can construct the detecting Lie subsuperalgebra $\ff$ as follows. The stable action 
ensures that there is a generic point $x_{0}$ in $\fg_{\1}$ (i.e., the $G_{\0}$-orbit of $x_{0}$ is closed and has maximal dimension). Set $H=\operatorname{Stab}_{G_{\0}}(x_0)$ and $N=\operatorname{Norm}_{G_{\0}}(H)$. Now let $\ff_{\1}=\fg_{\1}^{H}$ and set $\ff_{\0}=[\ff_{\1}, \ff_{\1}]$.  We then put $\ff=\ff_{\0}\oplus \ff_{\1}$.  We note that the $\ff$ used here differs slightly from the one in \cite{LNZ}.  However, the difference is only in the choice of $\ff_{\0}$ and one can verify that the arguments used in \emph{loc.\ cit.}  show that their results also apply to our choice of $\ff$.

When the action of $G_{\0}$ on $\fg_{\1}$ is a polar action (as defined by Dadok and Kac \cite{dadokkac}) 
we can construct a second detecting subalgebra $\fe$ as follows.  We fix a Cartan subspace $\fe_{\1} \subseteq \fg_{\1}$ and then obtain  
a classical Lie subalgebra $\fe=\fe_{\0}\oplus \fe_{\1}$ with $\fe_{\0}:=[\fe_{\1},\fe_{\1}]$.  From \cite[Section 8.9]{BKN1} it follows for $\mathfrak{gl}(m|n)$ and the simple classical Type I Lie superalgebras that $\fe_{\0}$ and $\ff_{\0}$ are tori and $[\fe_{\0}, \fe_{\1}]=[\ff_{\0}, \ff_{\1}]=0$.  For example, for $\mathfrak{gl}(n|n)$ one can take $\ff_{\1}$ to be the span of the matrix units $\left\{E_{i,n+i}, E_{n+i,i} \mid i=1, \dotsc, n \right\}$  and one can take $\fe_{\1}$ to be the span of  $\left\{E_{i,n+i}+ E_{n+i,i} \mid i=1, \dotsc, n \right\}$.  Then $\ff_{\0}$ is the Cartan subalgebra of $\mathfrak{gl}(n|n)$ consisting of diagonal matrices and $\fe_{\0}$ is the subalgebra spanned by $\left\{E_{i,i}+E_{n+i,n+i} \mid i = 1, \dotsc , n \right\}$.

Since $\fg$ is a classical Lie algebra such that $G_{\0}$ admits a stable and polar action on $\fg_{\1}$ we can make an appropriate choice of generic element $x_0$ to ensure that $\fe\leq \ff \leq \fg$. 
Furthermore, these inclusions induce restriction homomorphisms $S(\g_{\1}^{*}) \to S(\f_{\1}^{*})\to S(\e_{\1}^{*})$ which provide isomorphisms
\begin{equation} \label{eq:cohoiso}
\HH^{\bullet}(\fg,\fg_{\0};\C) \stackrel{\text{res}}{\longrightarrow}
\HH^{\bullet}(\ff,\ff_{\0};\C)^{N}\stackrel{\text{res}}{\longrightarrow} \HH^{\bullet}(\fe,\fe_{\0};\C)^{W}.
\end{equation}
where $W\subseteq G_{\0}$ is a pseudoreflection group. This implies that $R=\HH^{\bullet}(\fg,\fg_{\0};\C)$ 
is a polynomial algebra.  See \cite{BKN1} for further details.

\subsection{Superalgebras of the form \texorpdfstring{${\fz}={\ft}\oplus {\fz}_{\1}$}{z=t+z1}} \label{SS:z-properties} In this section we will focus on Lie superalgebras of the form 
$\fz=\fz_{\0}\oplus \fz_{\1}$ where $\fz_{\0}=\ft$ is a torus and $[\fz_{\0},\fz_{\1}]=0$.  This is a classical Lie superalgebra. As an example, $\fz$ could be one of the detecting subalgebras introduced in the previous section. The main goal of this section will be to 
classify thick tensor ideals in $\bK^{c}$ and to compute $\operatorname{Spc}(\bK^{c})$ where $\bK^{c}=\operatorname{Stab}(\FF_{(\fz,\fz_{\0})})$. We first verify that 
(\ref{SS:extend}.\ref{E:extendone1}), (\ref{SS:extend}.\ref{E:extendtwo2}) and  (\ref{SS:extend}.\ref{E:extendthree3}) hold for $\fz$ by following the arguments in \cite{BCR:96} and indicating the places where modifications are necessary. 

Let  $K$ denote an algebraically closed extension field of $\C$ having sufficiently large transcendence degree;  that is, larger than $\dim (\fz_{\1})$. For $M\in \CC_{(\fz,\fz_{\0})}$, set 
\begin{equation}
V^{r}_{\fz_{\1}}(K\otimes M) = \{\,x\in \operatorname{Proj}(K\otimes \fz_{\1})\mid K\otimes M \text{ is not projective as a $U({\langle x \rangle})$-module}\,\}.
\end{equation}
Here $\langle x \rangle$ is the Lie superalgebra generated by $x$ and here as elsewhere $\otimes$ denotes the tensor product over $\C$. 

We first verify the tensor product property for these generalized rank varieties.  This is an analogue of \cite[Theorem 7.4]{BCR:96}. 

\begin{prop} \label{P:TensorProductForz1Rank}
Let $M$ and $N$ be in $\CC_{(\fz,\fz_{\0})}$. Then 
$$V^{r}_{\fz_{\1}}(K\otimes (M\otimes N)) = V^{r}_{\fz_{\1}}(K\otimes M) \cap V^{r}_{\fz_{\1}}(K\otimes N).$$
\end{prop}

\begin{proof} Let $\fz, \fz^{\prime}$ be Lie superalgebras which satisfy the assumptions at the beginning of this section, and $M$ and $N$ be $\fz$- and $\fz^{\prime}$-modules, respectively. 
We see that 
$$
V^{r}_{\fz_{\1} \times \fz_{\1}{\prime}}(K\otimes(M\boxtimes N)) = V^{r}_{\fz_{\1}}(K\otimes M) \times V^{r}_{\fz^{\prime}_{\1}}(K\otimes N)
$$
as subvarieties of $\operatorname{Proj}(K\otimes (\fz_{\1} \times \fz_{\1}^{\prime})) = \operatorname{Proj}(K\otimes \fz_{\1}) \times \operatorname{Proj}(K\otimes \fz_{\1}^{\prime})$.  

The verification of this statement uses the same line of reasoning as given in \cite[Theorem 7.2]{BCR:96}. The only modification necessary is to note that for any  fixed $x \in \fz_{\1}$, the subalgebra $\langle x \rangle$ generated by $x$ is either isomorphic to $\fq (1)$ or a $1$-dimensional abelian Lie superalgebra concentrated in degree $\1$.  As a consequence, any (possibly infinite dimensional) $\langle x \rangle$-module $M$ in $\CC_{(\langle x \rangle,\langle x \rangle_{\0})}$ decomposes as a direct sum $M=\bigoplus M_{\lambda}$ of weight spaces for $\langle x \rangle_{\0}$ and this is a decomposition as $\langle x \rangle$-modules. If $\langle x \rangle_{\0} \neq 0$, then write $t$ for a vector which spans $\langle x \rangle_{\0}$.  We may then view $M_{\lambda}$ as a module for the Clifford algebra $A_{\la}=U(\fq(1))/{(t-\lambda(t))}$. Applying \cite[Corollary 4.8]{Aus}   we have that $M_{\lambda}$ is a direct sum of 
finite dimensional indecomposable modules. If $\lambda\ne 0$ then the only indecomposable modules are two dimensional projective simple modules. When $\lambda=0$ the indecomposables consist of the one dimensional trivial module  and its two dimensional projective cover (see \cite[Proposition 5.2]{BKN2}).  On the other hand, if $\langle x \rangle_{\0}=0$, then the enveloping algebra is isomorphic to an exterior algebra on the generator $x$ and one is immediately in the case $\lambda=0$. In every case one can then analyze the variety of the external product of such indecomposables as in 
\cite[Theorem 7.2]{BCR:96}. 

In order to finish the proof for a given $\fz$ one can use the coproduct map $\Delta:U({\fz})\rightarrow U(\fz)\otimes U(\fz)$ and then proceed as in the proof of \cite[Theorem 7.4]{BCR:96}. 
\end{proof} 

We can now prove a version of Dade's Lemma for modules in $\CC_{(\fz,\fz_{\0})}$.

\begin{prop} \label{T:Dade}
Let $M\in\CC_{(\fz,\fz_{\0})}$. Then $M$ is projective if and only if $V^{r}_{\fz_{\1}}(K\otimes M) = \varnothing$. 
\end{prop}

\begin{proof}
If $M$ is projective in $\CC_{(\fz,\fz_{\0})}$ then $K\otimes M$ is projective when restricted to $\langle x \rangle $ for any $x\in \Proj(K\otimes\fz_{\1})$ and so $V^{r}_{\fz_{\1}}(K\otimes M) = \varnothing$. 

For the converse, note that by our assumptions on $M$ and $\fz$ we have $M = \bigoplus_{\la\in\ft^{*}} M_{\la}$ where $M_{\la}$ is the $\lambda$-weight space for $\fz_{\0}=\ft$. Furthermore, this decomposition is as $\fz$-modules. Since the generalized rank variety of a direct sum is the union of the varieties of the direct summands it suffices to prove the result for $M_{\la}$. 
Now $M_{\la}$ is a module for the Clifford algebra $U(\fz)/(t-\lambda(t) \mid t\in\ft)$. By \cite[Section 5.2]{BKN2} the associated block ${\mathcal B}_{\la}$ containing the module $M_{\la}$ of $\CC_{(\fz,\fz_{\0})}$ has only a single simple module up to parity shift.  Call it $L(\la)$.

Now assume  $V^{r}_{\fz_{\1}}(K\otimes M_{\la}) = \varnothing$. Since 
$$V^{r}_{\fz_{\1}}(K\otimes M_{\la} \otimes L(-\la)) \subseteq V^{r}_{\fz_{\1}}(K\otimes M_{\la}) \cap V^{r}_{\fz_{\1}}(K\otimes L(-\la))$$ 
it follows that 
$V^{r}_{\fz_{\1}}(K\otimes M_{\la} \otimes L(-\la)) = \varnothing$. But, $M_{\la} \otimes L(-\la)$ has trivial $\ft$-action, so it is a module for $U(\fz/\fz_{\0}) \cong \Lambda^{\bullet}(\fz_{\1})$, where $\fz_{\1}$ is 
regarded as an abelian Lie superalgebra. The proof of Dade's Lemma (Theorem 5.2) in \cite{BCR:96} directly carries over 
to the algebra $\Lambda^{\bullet}(\fz_{\1})$ (which behaves like the group algebra of an elementary abelian 2-group). Hence, $M_{\la} \otimes L(-\la)$ is projective as a module for $\Lambda^{\bullet}(\fz_{\1})$.

It remains to show that $M_{\la}$ is projective in $\CC_{(\fz,\fz_{\0})}$. Recalling that $L(\la)$ is the only simple module in the block ${\mathcal B}_{\la}$ up to parity shift, and that $L(\la)$ is finite dimensional, it suffices to show that 
$$\Ext^{n}_{(\fz,\fz_{\0})}(L(\la),M_{\la})=\Ext^{n}_{(\fz,\fz_{\0})}(\C,M_{\la}\otimes L(-\la))=0$$ 
for all $n>0$. Apply the Lyndon-Hochschild-Serre spectral sequence for the pair $(\fz_{\0},\fz_{\0}) \triangleleft (\fz,\fz_{\0})$ (cf.\ \cite[I.\ Theorem 6.5]{BorelWallach}):
$$
E_{2}^{i,j} = \Ext^{i}_{(\fz_{\1},\{0\})}(\C,\Ext^{j}_{(\fz_{\0},\fz_{\0})}(\C,M_{\la}\otimes L(-\la))) \Rightarrow \Ext^{i+j}_{(\fz,\fz_{\0})}(\C,M_{\la}\otimes L(-\la)).
$$
The spectral sequence collapses because modules in $\CC_{(\fz_{\0},\fz_{\0})}$ are completely reducible. This yields 
$$
\Ext^{n}_{(\fz,\fz_{\0})}(\C,M_{\la}\otimes L(-\la)) \cong \Ext^{n}_{\fz_{\1}}(\C,M_{\la}\otimes L(-\la)) = 0
$$
for all $n>0$ by the last sentence of the preceding paragraph.
\end{proof}

We let $V_{(\fz , \fz_{\0})}(-)$ be the cohomological variety theory for $\fz$ as defined in \cref{SS:relativecohomology}.  As explained in the proof of \cref{T:torusclassification} (below), this is a support data and we can define the variety theory  $\VV_{(\fz,\fz_{\0})}(-)$ on $\bK$ following \cref{D:abstractsupport2}.  Assuming that the transcendence degree of $K$ over $\C$ is at least $\dim (\fz_{\1})$ we have a bijective correspondence between $V^{r}_{\fz_{\1}}(K)$ and $\V_{(\fz,\fz_{\0})}(\C)$, which equals $X$, by a direct calculation. Namely, a point in  $V^{r}_{\fz_{\1}}(K)$ corresponds to a generic point for a uniquely determined irreducible subvariety of $X$ (see \cite[Propositions 2.1, 3.1]{BCR:96}). Let ${\beta}^{*}:\V_{(\fz,\fz_{\0})}(\C)\rightarrow V^{r}_{\fz_{\1}}(K)$ denote this correspondence.  Under this correspondence one can use the same strategy as outlined in \cite[Sections 9--10]{BCR:96} to show that for 
all $M\in\CC_{(\fz,\fz_{\0})}$ 
$${\beta}^{*}:\V_{(\fz,\fz_{\0})}(M)\rightarrow V^{r}_{\fz_{\1}}(K\otimes M)$$
is a bijection.
For the analogous statement, see \cite[Theorem 10.5]{BCR:96}. In order to make this translation, 
for $W\in \XX_{sp}$, $E(W)$ corresponds to our $\Gamma_{W}(\unit )$ and $F(W)$ corresponds to $L_{W}(\unit )$.  Moreover, if $W\in \XX_{irr}$, 
$\kappa(W)$ corresponds to $\nabla_{W}(\unit )$ (cf.\ \cite[p.\ 609, Table]{BIK:08a}). Hence, we obtain the following result. 

\begin{prop} \label{P:beta*bijection}
Let $M\in\CC_{(\fz,\fz_{\0})}$. Then ${\beta}^{*}:\V_{(\fz,\fz_{\0})}(M) \rightarrow V^{r}_{\fz_{\1}}(K\otimes M)$ is a bijection. 
\end{prop}

We can now classify the thick tensor ideals of $\operatorname{Stab}(\FF_{(\fz,\fz_{\0})})$.

\begin{theorem}\label{T:torusclassification} Let $\fz=\fz_{\0}\oplus \fz_{\1}$ where $\fz_{\0}=\ft$ is a torus and $[\fz_{\0},\fz_{\1}]=0$. Set 
$R=S^{\bullet}(\fz_{\1}^{*})$. Then the thick tensor ideals of $\bK^{c}=\operatorname{Stab}(\FF_{(\fz,\fz_{\0})})$ are in bijective correspondence with specialization closed subsets of $\operatorname{Proj}(R)$.  Furthermore, 
$\operatorname{Spc}(\bK^{c})$ is homeomorphic to $\operatorname{Proj}(R)$. 
\end{theorem} 

\begin{proof} Since $\fz_{\0}$ acts trivially on $\fz_{\1}$ via the adjoint action, $R=S^{\bullet}(\fz_{\1}^{*})\cong \operatorname{H}^{\bullet}(\fz,\fz_{\0};\C)$. Set $\bK=\operatorname{Stab}(\CC_{(\fz,\fz_{\0})})$ 
and $\bK^{c}=\operatorname{Stab}(\FF_{(\fz,\fz_{\0})})$.  In this case the cohomological variety  $V_{(\fz,\fz_{0})}(-)$  of \cref{SS:relativecohomology}  has, by \cref{P:beta*bijection}, a concrete ``rank variety'' type description. So by applying \cref{P:TensorProductForz1Rank}, (\ref{SS:supportdata}.\ref{E:supportfive}) holds 
and hence $V_{(\fz,\fz_{\0})}$ is a support data. Moreover, by using (\ref{SS:supportdata}.\ref{E:supportfive}) and Carlson's $L_{\zeta}$ modules (see \cref{SS:Weight0} for another use of the same technique) one  can show that (\ref{SS:supportdata}.\ref{E:supporteight})  is satisfied. 

Since $\fz$ has only finitely many simple modules in each block, property (\ref{SS:supportdata}.\ref{E:supportseven}) holds by \cite[Theorem 2.9.1]{BKN3}. 
Given $V_{(\fz,\fz_{\0})}$ let $\VV_{(\fz,\fz_{\0})}$ be defined as in \cref{D:abstractsupport2}. Using $\beta^{*}$ to identify $\VV_{(\fz , \fz_{\0})}(M)$ with its rank variety as in \cref{P:beta*bijection}, one can transport the earlier results in this section to $\VV_{(\fz , \fz_{\0})}$ and
verify that (\ref{SS:extend}.\ref{E:extendone1}), (\ref{SS:extend}.\ref{E:extendtwo2}) and  (\ref{SS:extend}.\ref{E:extendthree3}) hold for $\fz$.  In particular, 
one can conclude that $\VV_{(\fz,\fz_{\0})}$ extends $V_{(\fz,\fz_{\0})}$. The results now follow from \cref{I:bijectiongeneral,K:bijectiongeneral}. 
\end{proof}

Since $R$ is also the cohomology ring for $\CC_{(\fz,\fz_{\0})}$ it follows that $\bK$ is an $R$-linear triangulated category in the sense of \cite{BIK:12}.  By \cite[Corollary 4.7]{BIK:12} the local global principle holds and by  \cref{SS:notationandconventions}, one has that 
$\bK$ is Noetherian in the sense of \cite[Definition 4.21]{BIK:12}.  We expect that $R$ stratifies $\bK$ and that the previous theorem can be deduced from \cite[Theorem 4.23]{BIK:12}.  

Note that one could also directly connect $\bK^{c}$ to the stable module category of finite-dimensional supermodules for the exterior algebra $\Lambda(\fz_{\1})$ (viewed as a superalgebra by declaring the generators to be odd) as follows.  Since $\fz_{\0}$ is a torus which commutes with $\fz$ the weight space decomposition of any $\fz$-supermodule is a decomposition as $\fz$-modules and there are corresponding decompositions $\CC_{(\fz,\fz_{\0})} = \oplus_{\lambda \in \fz^{*}_{\0}}\CC_{\lambda}$ and $\FF_{(\fz,\fz_{\0})} = \oplus_{\lambda \in \fz^{*}_{\0}}\FF_{\lambda}$.  Modules in the principal blocks $\CC_{0}$ and $\FF_{0}$ are annihilated by the ideal $I$ of $U(\fz)$ generated by $U(\fz_{\0})$ and so are naturally modules for the superalgebra $U(\fz )/I \cong \Lambda(\fz_{\1})$.  In this way we see that  $\CC_{0}$ (resp.\  $\FF_{0}$) is isomorphic to the category of all (resp.\ finite dimensional) supermodules for $\Lambda(\fz_{\1})$.  The same is true for their stable categories.  Furthermore, under this identification there is a bijection between the thick tensor ideals given by projecting an ideal $I$ of $\bK^{c}$ onto the principal block: $I \mapsto I \cap \bK^{c}_{0}$.  In this way the problem of identifying the Balmer spectrum and its support variety theory can be reduced to doing so for the superalgebra $\Lambda(\fz_{\1})$.  Such questions have been considered elsewhere (e.g.\ \cite{CarIyen, Ste, AAH}) in somewhat different settings.

\subsection{Support Data for Type I superalgebras} We will now investigate the construction of support data via cohomological support varieties for the classical Lie superalgebra $\fg$, its detecting subalgebras $\fe$ and $\ff$, and the relation between them. 

\begin{example} \label{ex:abelianLiesuper} Let $\fz=\fz_{\1}$ (i.e., $\fz_{\0}=0$) be an abelian Lie superalgebra consisting of odd degree elements.  Here $R=\HH^{\bullet}(\fz,\fz_{\0};\C)\cong S^{\bullet}(\fz_{\1}^{*})$. 
From Theorem~\ref{T:torusclassification} the thick tensor ideals of $\bK^{c}$ are in bijective correspondence with specialization closed subsets of $\operatorname{Proj}(R)$ and 
$\operatorname{Spc}(\bK^{c})$ identifies with $\operatorname{Proj}(R)$.
Note that this is analogous to the case of the group algebra of an elementary abelian $p$-group. 
\end{example} 

\begin{example}\label{ex:fesupport} Let $\fz$ denote the detecting subalgebra $\fe$ or $\ff$ of \cref{SS:detectingsubalgebras}.  In either case $\fz_{\0}$ acts trivially on $R\cong S^{\bullet}(\fz_{\1}^{*})$ and arguing as in the preceding example, we can classify thick tensor ideals in $\bK^{c}$ by specialization closed subsets of $\operatorname{Proj}(R)$, and $\operatorname{Spc}(\bK^{c})\cong \operatorname{Proj}(R)$. 

\end{example} 

\begin{example} \label{ex:cohosupport}  Let $\fg$ be a classical Lie superalgebra of Type I admitting both a stable and 
polar action of $G_{\0}$ on $\fg_{\1}$.  Let $R=\operatorname{H}^{\bullet}(\fg,\fg_{\0};\C)$. The ring homomorphisms given in \cref{eq:cohoiso} induce morphisms of topological spaces,
$$
{V}_{(\e,\e_{\0})}(\C)\stackrel{\text{res}^{*}}\longrightarrow {V}_{(\f,\f_{\0})}(\C)\stackrel{\text{res}^{*}}\longrightarrow {V}_{(\g,\g_{\0})}(\C),
$$
and isomorphisms (by passing to quotients),
$$
{V}_{(\e,\e_{\0})}(\C)/\!/W\stackrel{\text{res}^{*}}\longrightarrow {V}_{(\f,\f_{\0})}(\C)/\!/N\stackrel{\text{res}^{*}}\longrightarrow {V}_{(\g,\g_{\0})}(\C).
$$
Let $M$ be in $\bK^{c}=\operatorname{Stab}(\FF_{(\fg, \fg_{\0})})$. Then $\res^{*}$ induces maps
between the cohomologically defined pre-support data:
$${V}_{(\e,\e_{\0})}(M) \rightarrow {V}_{(\f,\f_{\0})}(M) \rightarrow {V}_{(\g,\g_{\0})}(M).$$ 
Since $M$ is a $G_{\0}$-module it follows that ${V}_{(\e,\e_{\0})}(M)$ (resp.\ $ {V}_{(\f,\f_{\0})}(M)$) is invariant under the action of $W$ (resp.\ $N$).  Thus we obtain induced embeddings  between the geometric quotients (cf.\  \cite[(6.1.3)]{BKN1}):
\begin{equation}\label{E:embeddings}
{V}_{(\e,\e_{\0})}(M)/\!/W \hookrightarrow {V}_{(\f,\f_{\0})}(M)/\!/N
\hookrightarrow {V}_{(\g,\g_{\0})}(M).
\end{equation}

%If we let $\tilde{\ff}=\ff_{\1} \oplus \operatorname{Lie}(H)$, then $\tilde{\ff}$ is the ``intermediate'' detecting subalgebra used in \cite{LNZ}.  The inclusion $\ff  \hookrightarrow \tilde{\ff}$ defines a restriction map $\res^{*}:V_{(\ff , \ff_{\0})}(M) \to V_{(\tilde{\ff}, \tilde{\ff}_{\0})}(M)$ for any module $M \in \FF_{(\tilde{\ff}, \tilde{\ff}_{\0})}$ and the argument used in the proof of \cite[Theorem 4.4.1]{LNZ} shows that this gives a homeomorphism for all such $M$.  Using this along with \cite[Theorems 4.1.1 and 5.1.1(b)]{LNZ} Furthermore, with the tensor product property holding one can verify that (\ref{SS:supportdata}.\ref{E:supporteight}) also holds for $V_{(\g,\g_{\0})}$ using Carlson's $L_{\zeta}$ modules (see \cref{SS:Weight0} for another use of the same technique).  

By \cite[Theorem 5.1.1(b)]{LNZ}  the maps in \cref{E:embeddings} are in fact homeomorphisms.  Therefore, we can use the fact that (\ref{SS:supportdata}.\ref{E:supportone})--(\ref{SS:supportdata}.\ref{E:supportsix}) hold for $V_{(\e,\e_{\0})}$ to see that they also hold for $V_{(\g,\g_{\0})}$ (cf.\ \cite[Section 5.2]{LNZ}).  By \cite[Section 5.2]{LNZ} property (\ref{SS:supportdata}.\ref{E:supporteight}) holds as well.  

The preceding arguments demonstrate that $V_{(\fg, \fg_{\0})}$ is indeed a support data for $\bK^{c}$. However, 
we cannot use $V_{(\fg,\fg_{\0})}$ to classify thick tensor ideals or compute the Balmer spectrum 
because property (\ref{SS:supportdata}.\ref{E:supportseven}) does not hold in general. For example, for any Kac module $K^{+}(\lambda)$ for $\fg=\mathfrak{gl}(m|n)$, $V_{(\fg,\fg_{\0})}(K^{+}(\lambda))=\varnothing$  by \cite[Corollary 3.3.1]{BKN2}, but it is well known that Kac modules need not be projective in $\FF_{(\fg,\fg_{\0})}$. 

We remark that $\bK$ is an $R$-linear triangulated category and the local global principal holds by \cite[Corollary 4.7]{BIK:12} but in general the minimality condition of \emph{loc.\ cit.}\ fails and, hence, $R$ fails to stratify $\bK$.  This can already be seen in the case of $\fg =\gl (1|1)$.
\end{example}

\begin{example}\label{ex:dufloserganova} Consider the associated variety for a classical Lie superalgebra $\fa$ defined by Duflo and Serganova \cite{dufloserganova}. 
Let $X=\operatorname{Proj}(\{x\in \fa_{\1} \mid [x,x]=0\})$, a projective variety over $\C$, and let $\XX_{cl}$ be the set of closed subsets of $X$. 
Define $V:\bK^{c}\rightarrow \XX_{cl}$ in the following way.  For each $M$ in $\bK^{c}$, let 
$$V(M):=X_{M}=\operatorname{Proj}\left(  \{ x\in X \mid \text{$M$ is not free as a $U(\langle x \rangle)$-module} \}\right).$$ 

The assignment $M \mapsto V(M)$ satisfies properties (\ref{SS:supportdata}.\ref{E:supportone})--(\ref{SS:supportdata}.\ref{E:supportsix}). In many instances, for example when $\fa=\gl(m|n)$, 
$V$ satisfies (\ref{SS:supportdata}.\ref{E:supportseven}). Moreover, one can check that (\ref{SS:supportdata}.\ref{E:supporteight}) holds for $\gl(1|1)$ and $\gl(2|2)$. These low rank examples also show that $V$ does not classify the thick tensor ideals of $\bK^{c}$.  From the point of view of our setup what fails is that $X$ is not a Zariski space. 
\end{example} 

\begin{example}\label{ex:supportsforstandardparabolics} Let $\fg=\mathfrak{gl}(m|n)$.  The algebra $\fg$ is Type I with $\fg=
\fg_{-1}\oplus \fg_{0} \oplus \fg_{1}$ where $\fg_{0}\cong \gl(m)\oplus \gl(n)$. 
Set $\fq^{+}=\fg_{0}\oplus \fg_{1}$ and note that it is a classical Lie superalgebra. We will investigate the cohomologically defined support data for this parabolic subalgebra, 
$\bK = \operatorname{Stab}\left(\CC_{(\fq^{+}, \fg_{0})} \right)$, and the classification of thick tensor ideals in $\bK^{c}=\operatorname{Stab}\left( \F_{(\fq^{+},\fg_{0})}\right)$.

We first consider the Lie subsuperalgebra $\fg_{1}$.  Let $R=\HH^{\bullet}(\fg_{1},\C)=\HH^{\bullet}(\fg_{1}, \{0 \};\C)\cong S^{\bullet}(\fg_{1}^{*})$ with $Y=\operatorname{Proj}(R)$. 
Since $\fg_{1}$ is an abelian Lie superalgebra consisting of odd degree elements one has a support data 
$V_{\fg_{1}}:\bL^{c}\rightarrow \YY_{cl}$ where $\bL$ is the stable module category of $\CC_{(\fg_{1}, \{0 \})}$ as 
in Example~\ref{ex:abelianLiesuper} and $\YY_{cl}$ is the collection of all closed subsets of $Y$. Moreover, let $\VV_{\fg_{1}}: \bL \to \YY$ be the extension of $V_{\g_{1}}$ as 
used in Theorem~\ref{T:torusclassification}. 

Using the support data $V_{\fg_{1}}$ define a support data $\widehat{V}:\bK^{c}\rightarrow \YY_{cl}$ by $\widehat{V}(M)=V_{\fg_{1}}(M)$. Properties 
(\ref{SS:supportdata}.\ref{E:supportone})--(\ref{SS:supportdata}.\ref{E:supportsix}) follow because $V_{\g_{1}}$ is a support data for $\bL^{c}$ and restriction is an exact functor. 
For (\ref{SS:supportdata}.\ref{E:supportseven}), one can use the fact that a module in $\CC_{(\fq^{+},\fg_{0})}$ is projective (equivalently, its complexity is zero)
if and only if it is projective over $U(\fg_{1})$ \cite[Theorem 3.3.1]{BKN4}; note that there only finite dimensional modules are considered but the relevant part of the proof does not use this fact. 
Thus one can define an extension $\widehat{\VV}:\bK\rightarrow \YY$ of $\widehat{V}$ by setting $\widehat{\VV}(M)=\VV_{\fg_{1}}(M)$. 

Now observe that the group $G_{0}=GL(m)\times GL(n)$  acts on $R$ and so we may consider $X=Y_{G_{0}}=G_{0}\text{-}\operatorname{Proj}(R)$.  
If $M\in \bK^{c}$ then $\widehat{V}(M)$ is a $G_{0}$-invariant closed set in $Y$ and thus the support data $\widehat{V}$ can be viewed as having image in $\XX_{cl}$ via the map $\rho: Y \to X$ given in \cref{SS:zariski}. 
In order to show that $\widehat{V}:\bK^{c}\rightarrow \XX_{cl}$ classifies thick tensor ideals we need to verify that (\ref{SS:supportdata}.\ref{E:supporteight}) holds. 
The number of $G_{0}$-orbits on $Y_{\text{max}}$ is finite with orbit representatives given by matrices $I_{t}$
in $\fg_{1}$ of rank $t$ with $0\leq t \leq r$ and $r=\text{min}(m,n)$. Furthermore, $\overline{G_{0}\cdot I_{s}}\subseteq
\overline{G_{0}\cdot I_{t}}$ if and only if $s\leq t$.  The closed $G_{0}$-invariant sets in $Y_{\text{max}}$ have the form $W_{t}=\overline{G_{0}\cdot I_{t}}$ for some $0\le t\leq r$.  

There certainly exists a module $N\in \FF_{(\fq^{+},\fg_{0})}$ 
such that $V_{\fg_{1}}(N) = W_{t}$; for example, one may take $N$ to be a simple module of atypicality $t$ (see \cite[(3.8.1)]{BKN3}). 
We can now conclude that $\Gamma$ and $\Theta$ yield a bijective correspondence between thick tensor ideals of $\operatorname{Stab}(\F_{(\fq^{+},\fg_{0})})$ and 
specialization closed $G_{0}$-invariant subsets of  $\operatorname{Proj}(S^{\bullet}(\g_{1}^{*}))$ (i.e., specialization closed subsets of $X$). 

Let $M_{t}$ be the maximal ideal in $R$ corresponding to the point $I_{t}$. 
According to \cite[Proposition 14]{Lorenzpaper}, $X=Y_{G_{0}}$ is finite and $X=\{\cap_{g} gM_{t}:\ 1\leq t \leq r\}$. Furthermore, irreducible 
sets of $X$ are ordered by inclusion: $\overline{\{\cap_{g}\ gM_{s} \}} \subseteq \overline{\{\cap_{g}\ gM_{t}\}}$ if and only if $s\leq t$. 
Consequently, every specialization closed set in $X$ is irreducible, thus all thick tensor ideals in $\operatorname{Stab}(\F_{(\fq^{+},\fg_{0})})$ are prime.   

An identical analysis applies also to $\fq^{-}= \fg_{0}\oplus \fg_{-1}$ using the support data $\widehat{V}(M)=V_{\fg_{-1}}(M)$.

\end{example}

\section{Classification of Thick Tensor Ideals for \texorpdfstring{$\mathfrak{gl}(m|n)$}{gl(m|n)}} \label{S:MainTheorem}

\subsection{}\label{SS:glmnsupportdata}  Let $\fg=\gl(m|n)$ and let $\bK = \operatorname{Stab}\left(\CC_{(\fg , \fg_{\0})} \right)$ with $\bK^{c}=  \operatorname{Stab}\left(\FF_{(\fg , \fg_{\0})} \right)$.  
Our goal is to construct a support data $\widehat{V}:\bK^{c}\rightarrow \XX_{cl}$. 

First, let $\bL$ (resp.\ $\bL^{c}$) be the stable module category for $\CC_{(\ff,\ff _{\0})}$ (resp.\ $\FF_{(\ff,\ff_{\0})}$).  Set 
\[
R=\HH^{\bullet}(\ff,\ff_{\0};\C )\cong S^{\bullet}(\ff_{\1}^{*}), 
\] $Y=\operatorname{Proj}(R)$, $\YY_{cl}$ to be the closed subsets of $Y$, and $\YY$ to be the set of all subsets of $Y$.  
From \cref{ex:fesupport} there exists a cohomologically defined support data $V:\bL^{c}\rightarrow \YY_{cl}$ with an extension $\VV :\bL \to \YY$.
  %It is given by  $$V(M)=\operatorname{Proj}(\operatorname{Spec}\left( R/J_{\tilde{\ff},M})\right)$$ for $M$ in $\bL^{c}$.  

Now define $\widehat{V}:\bK^{c}\rightarrow \YY_{cl}$ and $\widehat{\VV}:\bK\rightarrow \YY$ by setting $\widehat{V}(M)=V(M)$ and $\widehat{\VV}(M) = \VV (M)$, respectively.  Let  $N=\operatorname{Norm}_{G_{\0}}(\ff_{\1})$.  Then $N$ acts on $R$ and so 
set $X=N\text{-}\Proj(R)$ which is a Zariski space. 
Let $\XX_{cl}$ be the closed sets in $X$ and let $\XX$ be the collection of all subsets of $X$. Any $M$ in $\FF_{(\fg,\fg_{\0})}$ is a $G_{\0}$-module, and hence $V(M)$ will be an $N$-invariant closed set.  As a consequence we can and will view $\widehat{V}(M)$ as lying in $\XX_{cl}$ 
under the map $\rho:\operatorname{Proj}(R)\rightarrow N\text{-}\Proj(R)$ from \cref{SS:zariski}. That is, 
$\widehat{V}: \bK^{c} \to \XX_{cl}$. Similarly, one can view the image of $\widehat\VV$  as lying in $\XX$ under the map $\rho$, and hence consider $\widehat\VV$  as a map
\[
\widehat{\VV}: \bK \to \XX.
\]

The fact that $V$ is a support data shows that (\ref{SS:supportdata}.\ref{E:supportone})--(\ref{SS:supportdata}.\ref{E:supportsix}) hold for $\widehat{V}$.  The fact that $\VV$ extends $V$ immediately implies $\widehat{\VV}$ also satisfies (\ref{SS:supportdata}.\ref{E:supportone}), (\ref{SS:supportdata}.\ref{E:supporttwo}),  (\ref{SS:supportdata}.\ref{E:supportfour}), (\ref{SS:supportdata}.\ref{E:supportfive}), and \cref{D:extends}(ii).  Since the restriction functor is exact, the fact that $\VV$ satisfies (\ref{SS:supportdata}.\ref{E:supportthree}) implies that $\widehat{\VV}$ does as well.   Using \cite[Theorem 3.3.1]{LNZ} we see that a module $M\in \CC_{(\fg,\fg_{\0})}$ is projective if and only if its restriction is projective in $\CC_{(\ff,\ff_{\0})}$; note that in \emph{loc.\  cit.}\ they assume $M$ lies in $\FF_{(\fg , \fg_{\0})}$ but the proof of the needed results does not use this fact.  This along with the fact that $\VV$ satisfies (\ref{SS:supportdata}.\ref{E:supportseven}) implies it also holds for $\widehat{\VV}$.  In short, $\widehat{\VV}$ extends $\VV$.

To apply the results of \cref{S:classificationtheorems} it only remains to verify that $\widehat{V}$ satisfies (\ref{SS:supportdata}.\ref{E:supporteight}).  This verification involves results on geometric induction which will be obtained in the subsequent sections. Using the discussion in \cref{SS:zariski} and the rank variety description of $V$, we can reduce this problem to showing that for any $N$-invariant 
closed subvariety $W$ in $\operatorname{Proj}(\ff_{\1})$ there exists a module $M$ in $\FF_{(\fg,\fg_{\0})}$ such that 
$$W=\widehat{V}_{\max}(M)\cong V^{r}_{\f_{\1}}(M):= \operatorname{Proj}(\{x\in \ff_{\1}\mid \text{$M$ is not projective as a $U(\langle x \rangle)$-module}\}).$$
This will be proved in \cref{S:VarietiesandGeometricInduction} and the notation $V^{r}_{\f_{\1}}(M)$ will be used for the remainder of the paper.  

\subsection{} Assuming (\ref{SS:supportdata}.\ref{E:supporteight})  is verified for $\widehat{V}: \bK^{c} \to \XX_{cl}$, we can now state the main result.  Namely, $\widehat{V}$ provides a classification of the thick tensor ideals for $\FF_{(\gl(m|n), \gl(m|n)_{\0})}$ and $X$ provides a realization of the Balmer spectrum for this category. These results are consequences of our general setup in \cref{SS:classifying,SS:balmer}. 

\begin{theorem} \label{T:glmnclassification}
Let $\fg=\glmn$, let $\ff$ be the detecting subalgebra of $\fg$, and let $N=\operatorname{Norm}_{G_{\0}}(\ff_{\1})$. Then there is a bijection between
the set of thick tensor ideals of $\operatorname{Stab}(\F_{(\fg,\fg_{\0})})$
and the set of specialization closed subsets of $N\text{-}\Proj(S^{\bullet}(\ff_{\1}^{*}))$.
\end{theorem}

\begin{theorem}\label{T:glmnBalmer}
Let $\fg=\glmn$, let $\ff$ be the detecting subalgebra of $\fg$, and let $N=\operatorname{Norm}_{G_{\0}}(\ff_{\1})$. Then there is a homeomorphism between $\operatorname{Spc}(\bK^{c})$ and 
$N\text{-}\Proj(S^{\bullet}(\ff_{\1}^{*}))$.  In particular, there is a bijection between the set of  prime thick tensor ideals of $\operatorname{Stab}(\F_{(\fg,\fg_{\0})})$ and the collection of irreducible $N$-stable closed subsets of $\operatorname{Proj}(S^{\bullet}(\ff_{\1}^{*}))$.
\end{theorem}

Let us also mention that the analogues of the above theorems are also true for the type $C$ Lie superalgebra $\mathfrak{osp}(2|2n)$. It is a Type I Lie superalgebra of defect $1$ and with representation theory quite similar to that of $\mathfrak{gl}(1|1)$.  In particular, it has a detecting subalgebra which has all the requisite properties and, like $\mathfrak{gl}(1|1)$, realization can be achieved without parabolic induction by using only Carlson's $L_{\zeta}$ modules, Kac modules, and dual Kac modules.

\section{Geometric Induction}\label{S:GeometricInduction}

\subsection{}\label{SS:IntroGeometricInduction} In order to verify  (\ref{SS:supportdata}.\ref{E:supporteight}) for $\fg = \gl (m|n)$ we introduce modules which are constructed via geometric induction from parabolic subgroups.  Let $G$ denote the supergroup $\operatorname{GL}(m|n)$ and let $P$ be a parabolic subgroup of $G$.  By parabolic subgroup we will always mean the stabilizer of a flag of subsuperspaces in the natural representation of $G$.  In particular, define the \emph{standard parabolic} $Q=Q^{+}$ to be the subgroup of block upper triangular matrices with Levi subgroup isomorphic to $G_{\0}=\operatorname{GL}(m) \times \operatorname{GL}(n)$.  Set $Q^{-}$ to be the \emph{opposite standard parabolic} of block lower triangular matrices with Levi subgroup isomorphic to $G_{\0}$.

 We define induction from $P$-modules to $G$-modules using Lie superalgebras as in \cite{serganova6}.  See also \cite{GS} and references therein for further details on parabolic induction.  Let $\fp$ be the Lie subsuperalgebra of $\fg$ corresponding to $P$, let $\fu^{+}$ denote the nilpotent radical of $\fp$, $\fl$ the Levi subalgebra of $\fp$, and $\fu^{-}$ the opposite nilpotent radical.  Thus $\fp = \fl \oplus \fu^{+}$ and $\fg = \fu^{-} \oplus \fp$.  Note that the $\Z$-grading on $\fg$ is compatible with our choice of standard parabolic in the sense that for $\fp = \operatorname{Lie}(Q^{+})$ we have $\fg_{0}=\fl$, $\fg_{1}=\fu^{+}$, and $\fg_{-1}=\fu^{-}$.
    
Given an arbitrary $\fg$-module $M$ let 
\[
\Gamma (M)=\left\{\,m \in M \mid U(\fg_{\0})m \text{ is finite dimensional} \,\right\}.
\]  It is straightforward to see that $\Gamma$ is an endofunctor on the category of $\fg$-modules and that it coincides with the analogous functor defined by replacing $\fg_{\0}$ with $\fg$.  The functor $\Gamma$ is left exact and we let $\Gamma^{i}$ denote its $i$th right derived functor.  In particular, we define a functor from $P$-modules to $G$-modules by
\[
H^{i}(G/P, N) = \Gamma^{i}\left(\Hom_{U(\fp )}(U(\fg), N ) \right),
\]  
where $\Hom_{U(\fp )}(U(\fg), N )$ is a left $U(\fg)$-module via $(u.f)(u')=(-1)^{(\bar{f}+\bar{u'})\bar{u}} f(u'u)$.

As $\fg_{\0}$-modules these modules have the following useful alternative description.  For any finite dimensional $\fp$-module $N$ we have
\[
\Hom_{U(\fp )}\left(U(\fg ),  N^{*} \right) \cong \Hom_{U(\fp_{\0})}\left(U(\fg_{\0}), \left[\Lambda^{\bullet}(\fg_{\1}/\fp_{\1}) \otimes N \right]^{*} \right),
\] where the isomorphism is as $\fg_{\0}$-modules \cite[Section 4]{serganova6}.  Applying $\Gamma^{i}$ to both sides yields the following result.  

\begin{prop}\label{P:KeyFact}  Let $N$ be a finite dimensional $P$-module.  Then for all $i \geq 0$ there is an isomorphism of $\fg_{\0}$-modules
\[
H^{i}(G/P, N^{*}) \cong H^{i}(G_{\0}/P_{\0}, \left[ \Lambda^{\bullet}(\fg_{\1}/\fp_{\1}) \otimes N\right]^{*}).
\]
\end{prop}

\subsection{} These functors can be used to construct the simple modules for $\fg$ as follows.  Let $T$ be the subgroup of $G$ consisting of diagonal matrices and let $\ft=\operatorname{Lie}(T)$ be the corresponding Cartan subalgebra of $\fg_{\0}$ consisting of all diagonal matrices.   In particular, $\ft^{*}$ has a basis given by $\varepsilon_{1}, \dotsc , \varepsilon_{m+n}$ where $\varepsilon_{i} \in \ft^{*}$ is the functional which picks off the $i$th diagonal entry. Given $\lambda \in \ft^{*}$ we then may write $\lambda= \sum_{i=1}^{m+n}\lambda_{i}\varepsilon_{i}$.   Let $X^{+}_{0} \subseteq \ft^{*}$ be the set of dominant integral highest weights with respect to the Borel subalgebra $\fb_{\0}$ (of all upper triangular matrices) of $\fg_{\0}$.  For $\lambda \in X^{+}_{0},$ let $L_{0}(\lambda)$ be the simple finite dimensional $\fg_{\0}$-module of highest weight $\lambda$. 

Recall that $Q^{\pm}$ is our notation for the standard parabolic subgroups.  Set $\fq^{\pm}= \fg_{0} \oplus \fg_{\pm 1}= \operatorname{Lie}(Q^{\pm})$.  Since the Lie superalgebras $\fg_{\pm 1}$ are abelian ideals of $\fq^{\pm}$,  
$L_{0}(\lambda)$ can be viewed as a simple $\fq^{\pm}$-module via inflation.  For each $\lambda\in X^{+}_{0}$, the \emph{Kac modules} 
are constructed in the following way:
$$
K^{\pm}(\lambda)=U(\fg)\otimes_{U(\fq^{\pm})} L_{0}(\lambda). 
$$

From the PBW theorem for Lie superalgebras it follows that  
\begin{equation}\label{E:KacmoduleID}
H^{0}(G/Q^{\pm}, L_{0}(\lambda)^{*})^{*} \cong K^{\pm}(\lambda)
\end{equation} 
for any simple finite dimensional $\fq^{\pm}$-module $L_{0}(\lambda)$.

The module $K^{+}(\lambda)$ has a unique maximal submodule. The head of $K^{+}(\lambda)$ is the simple 
finite dimensional $\fg$-module $L(\lambda)$. Up to parity shift, the set $\{L(\lambda) \mid \lambda \in X_{0}^{+}\}$ 
constitutes a complete set of non-isomorphic simple modules in $\mathcal{F}_{(\g,\g_{\0})}$.  We disregard the parity shift as it will not play a role in what follows.

\subsection{}  Let $G=GL(m|n)$, let $G_{\pm 1}$ be the closed subgroup of $G$ with $\operatorname{Lie}(G_{\pm 1})=\fg_{\pm1}$, and let $P$ denote an arbitrary parabolic subgroup of $G$ with $P=L\ltimes U$, where $L$ is a Levi subgroup and $U$ is a unipotent subgroup. \emph{From this point on we always assume that $G_{1} \subseteq P$}. Let $Z$ be the closed subgroup $P\cap G_{-1}$ and $\mathfrak z$ be the Lie superalgebra associated to $Z$.

For a supergroup scheme $D$ with closed subgroup $C$ we have an induction functor from $C$-modules to $D$-modules, $\Ind_{C}^{D} -$, which can be defined as in the non-super setting (e.g.\ see \cite{zubkov}).  We use the notation and results of \cite{jantzen} extended to the supergroup setting.  This induction functor is 
left exact and admits higher right derived functors $R^{i}\Ind_{C}^{D} -$ for $i>0$.  When $D=G$ and $C=P$ these functors coincide with the functors $H^{i}(G/P, -)$ introduced in Section~\ref{SS:IntroGeometricInduction}.  Our first result relates the induction for the 
pairs $(G,P)$ and $(Q^{-},P_{0}\ltimes Z)$.  

\begin{prop} \label{P:equiv-vanishing}  Let $N$ be a module for $P$. Then for each $j>0$, $R^{j}\operatorname{ind}_{P}^{G} N=0$ if and only if 
$R^{j}\operatorname{ind}_{P_{0}\ltimes Z}^{Q^{-}}N=0$. 
\end{prop} 

\begin{proof} We first employ the spectral sequence: 
$$E_{2}^{i,j}=R^{i}\Ind_{Q^{-}}^{G} R^{j}\Ind_{P_{0}\ltimes Z}^{Q^{-}}N \Rightarrow R^{i+j}\Ind_{P_{0}\ltimes Z}^{G} N.$$
The functor $\Ind_{Q^{-}}^{G}(-)$ coincides with the coinduction functor $\Hom_{U(\fq^{-})}(U(\fg ), -)$.  From the PBW theorem $U(\fg )$ is projective as a $U(\fg^{-})$-module.  This implies the functor is exact and the spectral sequence collapses.  This yields the isomorphism: 
$$\Ind_{Q^{-}}^{G}[R^{j}\Ind_{P_{0}\ltimes Z}^{Q^{-}} N]\cong R^{j}\Ind_{P_{0}\ltimes Z}^{G} N$$ 
for $j>0$.  Since $\Ind_{Q^{-}}^{G}(-) \cong \Hom_{U(\fq^{-})}(U(\fg),-)$ we have that $\Ind_{Q^{-}}^{G}(-)$ takes 
non-zero modules to non-zero modules. We can now conclude that for $j>0$, 
$R^{j}\Ind_{P_{0}\ltimes Z}^{Q^{-}}N=0$ if and only if $R^{j}\Ind_{P_{0}\ltimes Z}^{G}N=0$. 
Next we involve the spectral sequence: 
$$E_{2}^{i,j}=R^{i}\Ind_{P}^{G}\ R^{j}\Ind_{P_{0}\ltimes Z}^{P}N \Rightarrow R^{i+j}\Ind_{P_{0}\ltimes Z}^{G} N.$$ 
Since $\Ind_{P_{0}\ltimes Z}^{P} (-)$ is exact, it follows that for $j\geq 0$,
$$R^{j}\Ind_{P_{0}\ltimes Z}^{G} N\cong R^{j}\Ind_{P}^{G} [\Ind_{P_{0}\ltimes Z}^{P} N].$$ 
Since $G_{1} \subseteq P$ we see, by applying (\ref{E:KacmoduleID}), that $K^{-}(0)^{*}$ is isomorphic to $\Ind_{P_{0}\ltimes Z}^{P} \C$ when restricted to $P$. 
Using this and the tensor identity twice yields
$$ R^{j}\Ind_{P}^{G} [\Ind_{P_{0}\ltimes Z}^{P} N] \cong R^{j}\Ind_{P}^{G} [(\Ind_{P_{0}\ltimes Z}^{P} \C)\otimes  N]\cong  K^{-}(0)^{*}\otimes R^{j}\Ind_{P}^{G} N.  
$$ 
These isomorphisms show for $j>0$ that $R^{j}\operatorname{ind}_{P}^{G} N=0$ if and only if $R^{j}\operatorname{ind}_{P_{0}\ltimes Z}^{Q^{-}}N=0$. 
\end{proof} 

\begin{prop} \label{P:cohomology-identity} Let $R=\operatorname{H}^{\bullet}(\fg_{-1},\C)$, $N$ be a module for $P$, and $M=\operatorname{ind}_{P}^{G}N$. If $R^{j}\operatorname{ind}_{P}^{G} N=0$ for all $j>0$ then 
$$\operatorname{Ext}^{\bullet}_{\fg_{-1}}(M,M)\cong \operatorname{Ext}^{\bullet}_{\fg_{-1}}(M,\operatorname{ind}_{P_{0}\ltimes Z}^{Q^{-}}N)$$ 
as $R$-modules. 
\end{prop} 

\begin{proof} First note that over ${\mathbb C}$ the category of rational $G$-modules (resp.\ $P$-modules) is equivalent to the relative categories 
${\mathcal C}_{({\mathfrak g},{\mathfrak g}_{0})}$ (resp.\ ${\mathcal C}_{({\mathfrak p},{\mathfrak p}_{0})}$). Since $G_{0}$ is reductive all rational 
$G_{0}$-modules are completely reducible. 

Set $M=\operatorname{ind}_{P}^{G}N$. For each $\lambda\in X_{0}^{+}$ let $m_{\lambda}=\dim L_{0}(\lambda)$. 
Then as $R$-modules 
\begin{eqnarray*} 
\operatorname{Ext}^{\bullet}_{\fg_{-1}}(M,M) &\cong & \operatorname{Hom}_{\C}(\C,\text{Ext}^{\bullet}_{\fg_{-1}}(M,M)) \\
&\cong & \operatorname{Hom}_{G_{0}}(\C,[\text{ind}_{{\C}}^{G_{0}}{\C}]\otimes \text{Ext}^{\bullet}_{\fg_{-1}}(M,M))\\
&\cong & \operatorname{Hom}_{G_{0}}(\C,{\mathbb C}[G_{0}]\otimes \text{Ext}^{\bullet}_{\fg_{-1}}(M,M)) \\
& \cong & \operatorname{Hom}_{\fg_{0}}(\oplus_{\lambda\in X_{0}^{+}}L_{0}(\lambda)^{m_{\lambda}}, 
\operatorname{Ext}^{\bullet}_{\fg_{-1}}(M,M)) \\
&\cong & \operatorname{Ext}^{\bullet}_{(\fq^{-},\fg_{0})}(\oplus_{\lambda\in X_{0}^{+}}L_{0}(\lambda)^{m_{\lambda}}\otimes M,M)\\
&\cong &  \operatorname{Ext}^{\bullet}_{(\fg,\fg_{0})}(\oplus_{\lambda\in X_{0}^{+}}K^{-}(\lambda)^{m_{\lambda}}\otimes M,M),
\end{eqnarray*} 
by the proof of \cite[Theorem 3.3.1]{BKN4} and Frobenius reciprocity. Note that  the coordinate algebra 
${\mathbb C}[G_{0}]$ is a $G_{0}\times G_{0}$-module isomorphic to $\oplus_{\lambda\in X_{0}^{+}}L_{0}(\lambda)\otimes L_{0}(\lambda)^{*}$. 
The $G_{0}$-module $\operatorname{Ext}^{\bullet}_{\fg_{-1}}(M,M)$ is completely reducible and the bimodule decomposition of 
 ${\mathbb C}[G_{0}]$ can be used to give a decomposition of $\operatorname{Ext}^{\bullet}_{\fg_{-1}}(M,M)$. 
% Here $K^{-}(\lambda)=U(\fg)\otimes_{U({\mathfrak q}^{-})} L_{0}(\lambda)$. 

We have a spectral sequence given by the composition of the Hom and induction functor (cf.\ \cite[I (4.5)]{jantzen}): 
$$E_{2}^{i,j}= \operatorname{Ext}^{i}_{(\fg,\fg_{0})}(\oplus_{\lambda\in X_{0}^{+}}K^{-}(\lambda)^{m_{\lambda}}\otimes M,R^{j}\Ind_{P}^{G}N)
\Rightarrow  \operatorname{Ext}^{i+j}_{(\fp,\fp_{0})}(\oplus_{\lambda\in X_{0}^{+}}K^{-}(\lambda)^{m_{\lambda}}\otimes M,N).
$$ 
Since $R^{j}\operatorname{ind}_{P}^{G} N=0$ for $j>0$ this spectral sequence collapses and yields the isomorphism: 
$$\operatorname{Ext}^{\bullet}_{(\fg,\fg_{0})}(\oplus_{\lambda\in X_{0}^{+}}K^{-}(\lambda)^{m_{\lambda}}\otimes M,M)\cong 
\operatorname{Ext}^{\bullet}_{(\fp,\fp_{0})}(\oplus_{\lambda\in X_{0}^{+}}K^{-}(\lambda)^{m_{\lambda}}\otimes M,N).$$ 
Since $G_{1}\subseteq P$ we have that $K^{-}(\lambda)\cong U(\fp)\otimes_{U(\fp_{0}\oplus \fz)}L_{0}(\lambda)$ as $\fp$-modules and  
\begin{eqnarray*} 
\operatorname{Ext}^{\bullet}_{\fg_{-1}}(M,M) &\cong& \operatorname{Ext}^{\bullet}_{(\fp,\fp_{0})}(\oplus_{\lambda\in X_{0}^{+}}K^{-}(\lambda)^{m_{\lambda}}\otimes M,N)\\
&\cong& \operatorname{Ext}^{\bullet}_{(\fp_{0}\oplus \fz, \fp_{0})}(\oplus_{\lambda\in X_{0}^{+}} L_{0}(\lambda)^{m_{\lambda}}\otimes M,N)\\
&\cong& \operatorname{Ext}^{\bullet}_{({\mathfrak q}^{-},\fg_{0})}(\oplus_{\lambda\in X_{0}^{+}} L_{0}(\lambda)^{m_{\lambda}}\otimes M,\Ind_{P_{0}\ltimes Z}^{Q^{-}}N)\\
&\cong& \operatorname{Ext}^{\bullet}_{\fg_{-1}}(M,\Ind_{P_{0}\ltimes Z}^{Q^{-}}N).
\end{eqnarray*} 
The second-last line holds by a spectral sequence argument similar to the one above after replacing $(G,P)$ with $(Q^{-},P_{0}\ltimes Z)$ and using Proposition~\ref{P:equiv-vanishing}.  The last line uses the proof of 
\cite[Theorem 3.3.1]{BKN4}.  
\end{proof} 

\subsection{} Consider the  pairs $(Q^{-},G_{-1})$ and $(P_{0}\ltimes Z,Z)$ with quotient groups 
$Q^{-}/G_{-1}\cong G_{0}$ and $(P_{0}\ltimes Z)/Z\cong P_{0}$. Applying \cite[I (6.12)]{jantzen} when $M$ is a $Q^{-}$-module and $N$ is a $P_{0}\ltimes Z$-module 
we have two spectral sequences which converge to the same abutment: 
\begin{equation}
\widehat E_{2}^{i,j}=\operatorname{Ext}^i_{\fg_{-1}}(M,R^{j} \Ind_{P_{0}\ltimes Z}^{Q^{-}} N)\Rightarrow (R^{i+j}{\mathcal G}_{1})(M,N),
\end{equation}
\begin{equation}
E_{2}^{i,j}=R^i\Ind^{G_{0}}_{P_{0}}\Ext^j_{\fz}(M,N)\Rightarrow
(R^{i+j}{\mathcal G}_{2})(M,N).
\end{equation} 
Here ${\mathcal G}_{1}=\text{Hom}_{{\mathfrak g}_{-1}}(M,\text{ind}_{P_{0}\ltimes Z}^{Q^{-}}(-))$ and 
${\mathcal G}_{2}=\text{ind}_{P_{0}}^{G_{0}}(\text{Hom}_{{\mathfrak z}}(M,-))$. These functors 
are used to construct the aforementioned spectral sequences with 
${\mathcal G}_{1}$ isomorphic to ${\mathcal G}_{2}$ (cf.\ \cite[I (6.12)]{jantzen}). 

In the case when $R^{j} \Ind_{P_{0}\ltimes Z}^{Q^{-}} N=0$ for all $j>0$, the first spectral sequence collapses and one can identify the abutment in the 
second spectral sequence. This yields: 
$$E_{2}^{i,j}=R^i\Ind^{G_{0}}_{P_{0}}\Ext^j_{\fz}(M,N)\Rightarrow \operatorname{Ext}^{i+j}_{\fg_{-1}}(M,\Ind_{P_{0}\ltimes Z}^{Q^{-}} N).$$
Now one can apply Proposition~\ref{P:cohomology-identity} when $M=\Ind_{P}^{G} N$ to rewrite the aforementioned spectral sequence as 
\begin{equation}\label{E:inversionspectral} 
E_{2}^{i,j}=R^i\Ind^{G_{0}}_{P_{0}}\Ext^j_{\fz}(M,N)\Rightarrow \operatorname{Ext}^{i+j}_{\fg_{-1}}(M,M).
\end{equation}

\subsection{A Version of the Borel-Weil-Bott Theorem} \label{SS:BorelWeilBott} We now specialize to the case when the parabolic subgroup is block upper triangular. \emph{Namely, for the remainder of the paper we will always have in mind a  fixed integer $0 \leq k \leq n$ and the fixed parabolic subgroup $P=P_{k}$ which is block upper triangular and has Levi subgroup isomorphic to $GL(m|n-k) \times GL(k)$.}  Note that the assumption from the previous section that $G_{1} \subseteq P$ is satisfied by this choice of parabolic subgroup.

 Let $\fp=\operatorname{Lie}(P)$ be the parabolic subalgebra of $\fg=\gl (m|n)$ corresponding to $P$ and let $\fl$ be the corresponding Levi subalgebra.

\begin{lemma}\label{L:BreakitDown}  Let $D > m(n-k)$ be fixed.  Let $L_{\fp}(\lambda)$ be a simple $\fp$-module with highest weight $\lambda$ satisfying $\lambda_{i}-\lambda_{i+1} > D$ for $i=1, \dotsc , m-1, m+1, \dotsc , m+n-1$.  Then as an $\fl_{\0}$-module 
\[
L_{\fp}(\lambda) \cong \bigoplus_{\gamma \in \Gamma_{\lambda}} L_{\fl_{\0}}(\gamma)
\] with $\Gamma_{\lambda} \subseteq \ft^{*}$ a finite set such that for all $\gamma \in \Gamma_{\lambda}$, $\gamma_{i}-\gamma_{i+1} > D - m(n-k)$ for $i=1, \dotsc , m-1, m+1, \dotsc , m+n-1$.
\end{lemma}

\begin{proof}   As $L_{\fp}(\lambda) =L_{\fl}(\lambda) = L_{\mathfrak{gl}(m|n-k)}(\lambda_{1}, \dotsc , \lambda_{m+n-k}) \boxtimes L_{\mathfrak{gl}(k)}(\lambda_{m+n-k+1}, \dotsc , \lambda_{m+n})$ and a simple $\gl(m|n-k)$-module is a quotient of the Kac module of the same highest weight, it suffices to prove that as $\fl_{\0}$-modules,
\begin{equation}\label{E:Kacmodulebreakdown}
 K^{+}_{\mathfrak{gl}(m|n-k)}(\lambda_{1}, \dotsc , \lambda_{m+n-k}) \boxtimes L_{\mathfrak{gl}(k)}(\lambda_{m+n-k+1}, \dotsc , \lambda_{m+n}) \cong  \bigoplus_{\gamma \in \widetilde{\Gamma}_{\lambda}} L_{\fl_{\0}}(\gamma)
\end{equation}
with $\widetilde{\Gamma}_{\lambda}$ satisfying the conditions stated in the theorem. This implies the desired result since $\Gamma_{\lambda} \subseteq \widetilde{\Gamma}_{\lambda}$.

From the PBW theorem for Lie superalgebras 
\[
 K^{+}_{\mathfrak{gl}(m|n-k)}(\lambda_{1}, \dotsc , \lambda_{m+n-k}) \cong  \Lambda^{\bullet}(\fl_{-1}) \otimes L_{\mathfrak{gl}(m|n-k)_{\0}}(\lambda_{1}, \dotsc , \lambda_{m+n-k}),
\] where the isomorphism is as $\mathfrak{gl}(m|n-k)_{\0} \cong \mathfrak{gl}(m) \oplus \mathfrak{gl}(n-k)$-modules.  From this we see that as a $\mathfrak{gl}(m) \oplus \mathfrak{gl}(n-k)$-module the Kac module $ K^{+}_{\mathfrak{gl}(m|n-k)}(\lambda_{1}, \dotsc , \lambda_{m+n-k})$ appears as a direct summand of  
\begin{equation*}
 \bigoplus_{t=0}^{\dim(\fl_{-1})=m(n-k)} \left(\fl_{-1} \right)^{\otimes t} \otimes L_{\mathfrak{gl}(m)\oplus \mathfrak{gl}(n-k)}(\lambda_{1}, \dotsc , \lambda_{m+n-k})  \cong \bigoplus_{\sigma} L_{\mathfrak{gl}(m|n-k)_{\0}} (\sigma).
\end{equation*} % The isomorphism is as  $\gl(m) \oplus \gl(n-k)$-modules. 
To describe the $\sigma$'s which appear in the direct sum, we observe that as a $\gl(m)\oplus \gl(n-k)$-module we have $\fl_{-1} \cong  V_{m}^{*} \boxtimes V_{n-k}$, where $V_{m}$ (resp.\ $V_{n-k}$) denotes the natural module for $\gl(m)$ (resp.\ $\gl(n-k)$).   A calculation using Pieri's formula \cite[(5.16)]{Macdonald} shows that the $\sigma$'s which appear in the above sum satisfy  $\sigma_{i}-\sigma_{i+1} > D - m(n-k)$ for $i=1, \dotsc , m-1, m+1, \dots, m+n-k-1$.

Inserting this description of the Kac module into the left hand side of \eqref{E:Kacmodulebreakdown} yields the right hand side and, hence, the desired result.
\end{proof}

\begin{lemma}\label{L:TensorDecomp}  Fix $D \geq km$ and let $L_{\fl_{\0}}(\gamma)$ be a simple $\fl_{\0}$-module with $\gamma_{i}-\gamma_{i+1} > D$ for $i=1, \dotsc , m-1,m+1, \dotsc, m+n-1$.  Then as $\fl_{\0}$-modules we have the following isomorphism:
\[
\Lambda^{\bullet}\left(\fu^{-}_{\1} \right) \otimes L_{\fl_{\0}}(\gamma) \cong \bigoplus_{\sigma \in \Sigma_{\gamma}} L_{\fl_{\0}}(\sigma)
\] for a finite set $\Sigma_{\gamma} \subset \ft^{*}$ such that for every $\sigma \in \Sigma_{\gamma}$ we have $\sigma_{i}-\sigma_{i+1} > D - km$ for $i=1, \dotsc , m-1,m+1, \dotsc, m+n-1$. 
\end{lemma}

\begin{proof}  As $\fl_{\0} \cong \mathfrak{gl}(m) \oplus \mathfrak{gl}(n-k) \oplus \mathfrak{gl}(k)$-modules we have that $\fu^{-}_{\1}$ is isomorphic to $V_{m}^{*} \boxtimes \C \boxtimes V_{k}$, where, as before, $V_{q}$ denotes the natural module of $\mathfrak{gl}(q)$.  Moreover,  
\[
L_{\fl_{\0}}(\gamma) = L_{\mathfrak{gl}(m)}(\gamma_{1}, \dotsc , \gamma_{m}) \boxtimes L_{\mathfrak{gl}(n-k)}(\gamma_{m+1}, \dotsc , \gamma_{m+n-k}) \boxtimes L_{\mathfrak{gl}(k)}(\gamma_{m+n-k+1}, \dotsc , \gamma_{m+n}).
\]  The claim follows as in the proof of the previous lemma using Pieri's formula \cite[(5.16)]{Macdonald}.
\end{proof}

We can now prove a coarse version of the parabolic Borel-Weil-Bott Theorem in the super setting for sufficiently dominant weights. Note we write it using the bound $d$ as it will be needed in what follows.

\begin{theorem}\label{T:SuperBWB}  Fix $d \geq 0$.  Let $\lambda \in X_{0}^{+}$ with $\lambda_{i}-\lambda_{i+1} >  d + mn$ for $i=1, \dotsc , m-1, m+1, \dotsc , m+n-1$.  Then 
\[
H^{j}(G/P, L_{\fp}(\lambda)^{*}) =  0
\] for $j>0$.  Furthermore, as $\fg_{\0}$-modules
\[
H^{0}(G/P, L_{\fp}(\lambda)^{*}) \cong \bigoplus_{\gamma \in \Gamma_{\lambda}}\bigoplus_{\sigma \in \Sigma_{\gamma}} L_{\fg_{\0}}(\sigma)^{*}
\] where $\Gamma_{\lambda}$ and $\Sigma_{\gamma}$ are as in the previous lemmas and, in particular, every $\sigma$ which appears in the direct sum satisfies $\sigma_{i}-\sigma_{i+1} > d$ for $i=1, \dotsc , m-1,m+1, \dotsc , m+n-1$.
\end{theorem}
 
\begin{proof} For the statement of the theorem it suffices to compute the induced module as a $\fg_{\0}$-module.  By Proposition~\ref{P:KeyFact},
\[
H^{j}(G/P, L_{\fp}(\lambda)^{*}) \cong  H^{j}(G_{\0}/P_{\0}, \left[ \Lambda^{\bullet}(\fg_{\1}/\fp_{\1}) \otimes L_{\fp}(\lambda)\right]^{*}).
\]  From our previous two lemmas we have an isomorphism as $\fl_{\0}$-modules: 
\[
\left[\Lambda^{\bullet}(\fg_{\1}/\fp_{\1}) \otimes L_{\fp}(\lambda) \right]^{*} \cong   \bigoplus_{\gamma \in \Gamma_{\lambda}} \bigoplus_{\sigma \in \Sigma_{\gamma}} L_{\fl_{\0}}(\sigma)^{*}
\] where each $\sigma$ satisfies $\sigma_{i}-\sigma_{i+1} > d $ for $i=1, \dotsc , m-1, m+1, \dotsc , m+n-1$.  In particular, each $\sigma$ is dominant regular.  That is, as a $\fp_{\0}$-module $\left[\Lambda^{\bullet}(\fg_{\1}/\fp_{\1}) \otimes L_{\fp}(\lambda) \right]^{*}$ has a composition series consisting of simple modules with dominant regular highest weights.  The theorem then follows by using the non-super parabolic Borel-Weil-Bott Theorem and an induction on the length of the composition series of modules with such a composition series.
\end{proof}

\subsection{The Parabolic Grading}\label{SS:gradings} 

Fix a positive real number $a$ and let $h \in T$  be the diagonal matrix 
\[
h = \operatorname{diag}(a, \dotsc ,a,a^{-1}, \dotsc ,a^{-1})
\]
where there are $m+n-k$ $a$'s and $k$ $a^{-1}$'s.  The adjoint action of $h$ on $\fg$ defines a $\Z$-grading compatible with our choice of parabolic.  That is, $h$ acts by $a^{-2}$, $1$, and $a^{2}$, respectively, on $\fu^{-}$, $\fl$, and $\fu^{+}$ and with the $\Z$-degree given by the exponent.  This 
element can be used to give a compatible $\Z$-grading to any weight module where the degree of a weight vector is determined by its weight.  Namely, for any integral weight $\mu =\sum_{i=1}^{m+n}\mu_{i}\varepsilon_{i}$ the degree is $\sum_{i=1}^{m+n-k}\mu_{i} - \sum_{i=m+n-k+1}^{m+n}\mu_{i}$ and we write $\deg(\mu)$ for this degree.  We call this the \emph{parabolic grading} to distinguish it from the $\Z_{2}$-grading and the $\Z$-grading introduced in Section~\ref{SS:IntroGeometricInduction} (although the two $\Z$-gradings coincide when $k=n$). 

Given a weight module $M$ we write $M_{t}$ for the $t$-th parabolic graded summand of $M$.
For example, the parabolic Verma module $U(\fg ) \otimes_{U(\fp )} L_{\fp}(\lambda)$ has a parabolic grading and by the PBW theorem it is non-zero only in degrees less than or equal to $\deg (\lambda)$.  As this module surjects onto the simple module $L_{\fg}(\lambda)$, it follows that the non-zero degrees of $L_{\fg}(\lambda)$ are also bounded above by $\deg (\lambda)$.  Similarly, since there exists a surjective $\fg$-module homomorphism by \cite[Lemma 2]{GS},
\begin{equation}\label{E:parabolicquotientmap}
U(\fg ) \otimes_{U(\fp )} L_{\fp}(\lambda) \to H^{0}(G/P, L_{\fp}(\lambda)^{*})^{*},
\end{equation}
it follows that the non-zero degrees of $H^{0}(G/P, L_{\fp}(\lambda)^{*})^{*}$ are bounded above by $\deg (\lambda)$.

The parabolic grading, \eqref{E:parabolicquotientmap} and Theorem~\ref{T:SuperBWB}  will be used to obtain a partial description of $H^{0}(G/P, L_{\fp}(\lambda)^{*})^{*}$ when $\lambda$ is sufficiently dominant.  To do so we will need the following lemma.

\begin{lemma}\label{L:purelyevengradings}  Fix $d \ge 0$.  Let $L_{\fg_{\0}}(\tau)$ be a finite dimensional $\fg_{\0}$-module of highest weight $\tau$ satisfying $\tau_{i}-\tau_{i+1}>d$ for $i=1, \dotsc , m-1, m+1, \dotsc , m+n-1$.  Then as $\fl_{\0}$-modules
\[
S^{t}(\fg_{\0}/\fp_{\0}) \otimes L_{\fp_{\0}}(\tau) \cong \left(U(\fg_{\0}) \otimes_{U(\fp_{\0})} L_{\fp_{\0}}(\tau) \right)_{\deg(\tau)-t} \cong  L_{\fg_{\0}}(\tau)_{\deg(\tau)-t}
\] for $t=0, \dotsc , d$.  

\end{lemma}
\begin{proof}  Let $M=U(\fg_{\0}) \otimes_{U(\fp_{\0})} L_{\fp_{\0}}(\tau)$ be the parabolic Verma $\fg_{\0}$-module of highest weight $\tau$.  Then one has a surjective map $M \to L_{\fg_{\0}}(\tau)$ which preserves the parabolic grading.  Using the PBW theorem for Lie algebras and the description of the left ideal which is the annihilator of the highest weight vector of $L_{\fg_{\0}}(\tau)$ (see \cite[Section 21.4, Proposition 23.2]{humphreysLieAlg}), we see that the graded surjective map is actually an isomorphism for degrees $\deg(\tau), \dotsc , \deg(\tau) - d$. This implies the second isomorphism and the first is from the PBW theorem.
\end{proof}

We can now prove that for sufficiently dominant weights the map \eqref{E:parabolicquotientmap} is an isomorphism for the degrees which are sufficiently close to $\deg (\lambda)$.  In the statement of the theorem we use the notation $S^{t}_{\text{super}}(\fg /\fp )$ to denote the $t$-th supersymmetric power of $\fg /\fp$.

\begin{theorem}\label{T:grading} Fix $d \geq 0$ and let $L_{\fp}(\lambda)$ be a finite dimensional simple $\fp$-module with $\lambda_{i}-\lambda_{i+1} > D=d+mn$ for $i=1, \dotsc , m-1,m+1, \dotsc , m+n-1$.  Then 
\[
S_{\rm{super}}^{t}(\fg/\fp) \otimes L_{\fp}(\lambda) \cong \left[ U(\fg ) \otimes_{U(\fp )} L_{\fp}(\lambda) \right]_{\deg(\lambda)-t} \cong \left[ H^{0}(G/P, L_{\fp}(\lambda)^{*})^{*}\right]_{\deg(\lambda)-t}
\] for $t = 0, \dotsc , d$.
\end{theorem}

\begin{proof}  
Fix $t \in \{ 0, \dotsc , d \}$.  Combining \cref{T:SuperBWB} and Lemma~\ref{L:purelyevengradings} we have that 
\begin{align*}
 \left[H^{0}(G/P, L_{\fp}(\lambda)^{*})^{*} \right]_{\deg(\lambda)-t} &\cong \bigoplus_{\gamma \in \Gamma_{\lambda}}\bigoplus_{\sigma \in \Sigma_{\gamma}} L_{\fg_{\0}}(\sigma)_{\deg(\lambda)-t} \\
                                      & \cong  \bigoplus_{\gamma \in \Gamma_{\lambda}}\bigoplus_{\sigma \in \Sigma_{\gamma}}  \left[S^{\bullet}(\fg_{\0}/\fp_{\0}) \otimes L_{\fp_{\0}}(\sigma) \right]_{\deg(\lambda)-t}.
\end{align*}  

Let us briefly explain the second line.  We first observe that in the parabolic grading $L_{\fg_{\0}}(\mu)$ is non-zero only in degrees less than or equal to $\deg (\mu)$ for any dominant $\mu$.  Furthermore, \eqref{E:parabolicquotientmap} implies that if $L_{\fg_{\0}}(\sigma)$ is a composition factor of $H^{0}(G/P, L_{\fp}(\lambda)^{*})^{*}$, then it is a composition factor of $U(\fg ) \otimes_{U(\fp )}L_{\fp}(\lambda)$.  Taken together this implies that if $L_{\fg_{\0}}(\sigma)$ appears in the first line then we must have that $\deg (\sigma) \leq \deg (\lambda)$ and it can contribute only to degrees less than or equal to $\deg (\sigma)$.

Now since $\sigma_{i}-\sigma_{i+1}>d$ for $i=1, \dotsc , m-1,m+1,\dotsc , m+n-1$ for all $\sigma$ which appear in the first line it follows that Lemma~\ref{L:purelyevengradings} can be applied to the degrees $\deg(\sigma), \dotsc , \deg(\sigma)-d$.  That is, it holds for $d$ degrees down from $\deg (\sigma)$ for each $\sigma$. Taken together with the previous paragraph we see that the above isomorphism holds for  $d$ degrees down from $\deg (\lambda)$.

Now, to prove the stated theorem we merely need to unwind the above isomorphisms using Lemmas~\ref{L:BreakitDown}~and~\ref{L:TensorDecomp}:
\begin{align*}
 \left[H^{0}(G/P, L_{\fp}(\lambda)^{*})^{*} \right]_{\deg(\lambda)-t} 
                                      & \cong  \bigoplus_{\gamma \in \Gamma_{\lambda}}\bigoplus_{\sigma \in \Sigma_{\gamma}}  \left[S^{\bullet}(\fg_{\0}/\fp_{\0}) \otimes L_{\fp_{\0}}(\sigma) \right]_{\deg(\lambda)-t} \\
                                      & \cong \bigoplus_{\gamma \in \Gamma_{\lambda}}\left[ S^{\bullet}(\fg_{\0}/\fp_{\0}) \otimes \left(\bigoplus_{\sigma \in \Sigma_{\gamma}} L_{\fp_{\0}}(\sigma) \right)\right]_{\deg(\lambda)-t} \\
 &  \cong  \bigoplus_{\gamma \in \Gamma_{\lambda}}\left[ S^{\bullet}(\fg_{\0}/\fp_{\0}) \otimes \Lambda^{\bullet}\left(\fu^{-}_{\1} \right) \otimes L_{\fl_{\0}}(\gamma)  \right]_{\deg(\lambda)-t} \\
&  \cong  \left[ S^{\bullet}(\fg_{\0}/\fp_{\0}) \otimes \Lambda^{\bullet}\left(\fu^{-}_{\1} \right) \otimes \left(\bigoplus_{\gamma \in \Gamma_{\lambda}}  L_{\fl_{\0}}(\gamma)  \right)\right]_{\deg(\lambda)-t} \\
& \cong  \left[ S^{\bullet}(\fu_{\0}^{-}) \otimes \Lambda^{\bullet}\left(\fu^{-}_{\1} \right) \otimes L_{\fp}(\lambda)\right]_{\deg(\lambda)-t} \\
& \cong  \left[U(\fg ) \otimes_{U(\fp )} L_{\fp}(\lambda) \right]_{\deg(\lambda)-t}.\qedhere
\end{align*} 
 \end{proof}

\section{Support Varieties}\label{S:VarietiesandGeometricInduction}

\subsection{}\label{SS:varietynotations}  We are now prepared to compute the support varieties of various $\fg$-modules. This will enable us to verify that in this setting (\ref{SS:supportdata}.\ref{E:supporteight}) holds for the support data $\widehat{V}$ of \cref{S:MainTheorem}.

To proceed we need to set notation.  Fix $\fg = \gl (m|n)$ and write $\ff$ for the detecting subalgebra of $\fg$.  Since $\gl (m|n) \cong \gl (n|m)$ we can and will assume that $m \leq n$ from this point on.  In particular $\ff_{\1}$ has a basis given by the matrix units $$e_{1, 2m}, e_{2, 2m-1}, \dotsc , e_{m, m+1}, e_{m+1,m}, \dotsc , e_{2m,1}.$$    Let $T$ denote the torus of $G_{\0}$ consisting of diagonal matrices.   Then $T$ acts on $\fg$ by the adjoint action and, in particular, the above basis for $\ff_{\1}$ consists of weight vectors.

As mentioned \cref{S:MainTheorem}, the underlying geometry in this setting is given by the spectrum of the cohomology ring of $\ff$; that is, the ring
\[
R:=\HH^{\bullet}(\ff ,\ff_{\0};\C ) \cong S^{\bullet}(\ff_{\1}^{*})=\C[\ff_{\1}]\cong \C [X_{1},X_{2},\dots, X_{m},Y_{1},Y_{2},\dots,Y_{m}],
\] where $X_j$ and $Y_j$ are defined below. In particular, 
\[
V_{({\mathfrak f},{\mathfrak f}_{\0})}(\C)_{\max}\cong V^{r}_{\ff_{\1}}(\C ) = \operatorname{Proj}\left(\operatorname{MaxSpec}(R) \right) = \operatorname{Proj}(\ff_{\1}).
\]

We set coordinates by identifying $R$ with the polynomial ring on $2m$ variables.  This is done by letting  $X_{j}: \ff_{\1} \to \C$ be the linear functional given on our basis of matrix units for $\ff_{\1}$ by $X_{j}\left({e_{m-j+1, m+j}} \right)=1$ and otherwise zero.  Similarly,  we define  $Y_{j}: \ff_{\1} \to \C$ to be the linear functional given on our basis by $Y_{j}\left({e_{m+j, m-j+1}}\right)=1$ and otherwise zero.  %Under the coadjoint action of $T$ the element $X_{j}$ (resp.\  $Y_{j}$) has $T$-weight $\alpha_{j}$ (resp.\  $-\alpha_{j}$). ***are these weights correct?***Are X and Y switched?***Set notation for glmn earlier***
 
Let $\Sigma_{m}$ denote the symmetric group on $m$ letters embedded diagonally in $G_{\0} \cong \operatorname{GL}(m) \times \operatorname{GL}(n)$ as permutation matrices so that the action of $\Sigma_{m}$ on $R$ is via simultaneous permutation of $X_{1}, \dotsc , X_{m}$ and $Y_{1}, \dotsc , Y_{m}$.  Let $N=\operatorname{Norm}_{G_{\0}}(H)$ as in \cref{SS:detectingsubalgebras} or, equivalently, $N=\operatorname{Norm}_{G_{\0}}(\ff_{\1})$.  A direct calculation (see \cite[Section 8.11]{BKN1} for the analogous calculation for $\mathfrak{sl}(n|n)$) verifies that $N=\Sigma_{m}TK$ where $K\cong \operatorname{GL}(n-m)$ is the subgroup of $G$ consisting of $(n-m) \times (n-m)$ invertible matrices in the lower right corner of $G$.  It is convenient to redefine
\[
N=\Sigma_{m}T
\] and instead consider the action of this group on $\ff_{\1}$.  Since $K$ fixes $R$ pointwise there is no harm in doing this.
 
As explained in \cref{SS:glmnsupportdata}, to verify that (\ref{SS:supportdata}.\ref{E:supporteight}) holds for the support data $\widehat{V}$ it suffices to prove that we can realize every $N$-invariant closed subvariety of $V^{r}_{\ff_{\1}}(\C) \cong \Proj (\ff_{\1})$ as $V^{r}_{\ff_{\1}}(M)$ for some $M \in \FF_{(\fg,\fg_{\0})}$.  Let $W$ be such a variety.  Since $\Sigma_{m}$ is a finite group we may write 
\[
W=\Sigma_{m}V
\] for some $T$-invariant closed subvariety $V$ of $\Proj (\ff_{\1})$.  Furthermore, since the action of $\Sigma_{m}$ distributes across unions we may use this along with (\ref{SS:supportdata}.\ref{E:supporttwo}) to assume without loss that $V$ is irreducible among the $T$-invariant subvarieties.  That is, if  we write $Z(I)$ for the closed subvariety of $\Proj (\ff_{\1})$ determined by a homogeneous ideal $I$ of $R$, it suffices to consider the case when $V=Z(I)$ for a homogeneous $T$-prime ideal $I$.  But since $T$ is connected it follows from \cite[Proposition 19]{Lorenzpaper2} that $I$ is a homogeneous $T$-invariant prime ideal in $R$.  Therefore, $V$ must be of the form   
\begin{equation}\label{E:generalV}
V= Z(X_{a_{1}}, \dotsc , X_{a_{s}}, Y_{b_{1}}, \dotsc , Y_{b_{t}}, g_{1}, \dotsc , g_{r})
\end{equation}
where $g_{1}, \dotsc , g_{r}$ are homogeneous polynomials of weight zero for $T$.   In short, to verify (\ref{SS:supportdata}.\ref{E:supporteight}) in this setting we  must prove that for any $V$ as in \cref{E:generalV} there exists an $M \in \FF_{(\fg , \fg_{\0})}$ such that $V^{r}_{\ff_{\1}}(M)= \Sigma_{m}V$.  This will be accomplished in the following sections. 

Before continuing it is convenient to introduce the following notation. Let $s, t, p$ be non-negative integers with $m\geq s, t \geq p \geq 0$. Set
\begin{equation}\label{E:vanishingcoordinates}
V(s,t,p) = Z(X_{1},\dots,X_{s},Y_{s-p+1},\dots,Y_{s-p+t} ).
\end{equation}
This is the variety given by the vanishing of $s$ $X$ coordinates, $t$ $Y$ coordinates, with $p$ ``overlapping pairs'' (that is, there are precisely $p$ indices $i$ such that $X_{i}$ and $Y_{i}$ are both required to vanish).  It is useful to note that since we may replace $V$ by its conjugate under the action of $\Sigma_{m}$, the case $r=0$ in \cref{E:generalV} will be complete once we realize $\Sigma_{m} V(s,t,p)$. 

\subsection{Weight Zero Polynomials} \label{SS:Weight0} We first  consider the case when $s=t=0$ in \eqref{E:generalV}.  That is, we show how to realize 
\[
\Sigma_{m} Z(g_{1},g_{2},\dots,g_{r}),
\] where $g_{1}, \dotsc , g_{r}$ are homogeneous weight zero polynomials on $\ff_{\1}$, as the support of a module in $\FF_{(\fg , \fg_{\0})}$.  This will be accomplished using Carlson's $L_{\zeta}$ modules.  
These modules are the standard tool used to prove realization in the setting of finite groups and are all that is needed there to prove realization in that setting.

\begin{prop}  Let $g_{1},g_{2},\dots,g_{r}$ be homogeneous polynomials in $R$ of weight zero with respect to $T$. Then there exists a module $M$ in $\FF$ such that 
$V^{r}_{\ff_{\1}}(M)=\Sigma_{m} Z(g_{1},g_{2},\dots,g_{r})$. 
\end{prop} 

\begin{proof} We show that we can reduce to the case when $g_{1}, \dotsc , g_{r}$ are $N=\Sigma_{m}T$-invariant polynomials.   In order to do so, we first show that $\Sigma_{m}Z(g_{1},g_{2},\dots,g_{r})$ is the zero set for a finite $\Sigma_{m}$-invariant collection of weight zero polynomials.

Set $q=m!$.   Let 
\[
\Gamma = \left\{ \sigma = (\sigma_{1}, \dotsc , \sigma_{q}) \in \Sigma_{m} \times \dotsb \times \Sigma_{m} \mid \sigma_{i} \neq \sigma_{j} \text{ for } 1 \leq i \neq j \leq q \right\}.
\]  Then $\Gamma$ acts diagonally on the set of monomials of degree $q$ in $g_{1}, \dotsc , g_{r}$ by
\[
\sigma.(g_{1}^{m_{1}}g_{2}^{m_{2}}\dots g_{r}^{m_{r}})=(\sigma_{1}g_{1})\dots (\sigma_{m_{1}}g_{1})(\sigma_{m_{1}+1}g_{2})\dots 
(\sigma_{m_{1}+\dots +m_{r-1}+1}g_{r})\dots (\sigma_{q}g_{r}).
\]  

We set 
\[
Y=\left\{ \sigma.(g_{1}^{m_{1}}g_{2}^{m_{2}}\dots g_{r}^{m_{r}}) \mid \sigma \in \Gamma, m_{1}, \dotsc , m_{r} \in \Z_{\geq 0} \text{ such that } \textstyle{\sum_{i}} m_{i} = q \right\}
\] and note that $Y$ is a finite set which is invariant under the action of $\Sigma_{m}$.
Furthermore, we have %)\dots (\sigma_{m_{1}+m_{2}+1}g_{2})
\begin{equation}\label{E:sigmatinvariants}
\Sigma_{m}Z(g_{1}, \dotsc , g_{r})  = Z\left( Y \right).
\end{equation}
To see this, we first let $b \in \ff_{\1}$ be an element of the left hand set.  Then $b = \tau a$ for some $\tau \in \Sigma_{m}$ and $a \in Z(g_{1}, \dotsc , g_{r})$.  Computing 
\begin{align*}
\left( \sigma.(g_{1}^{m_{1}}g_{2}^{m_{2}}\dots g_{r}^{m_{r}})\right)(b) &= \sigma.(g_{1}^{m_{1}}g_{2}^{m_{2}}\dots g_{r}^{m_{r}})(\tau a) \\
   &= (\sigma_{1}g_{1})\dots (\sigma_{q}g_{r})(\tau a) \\
   & = g_{1}(\sigma_{1}^{-1}(\tau a))\dotsb g_{r}(\sigma_{q}^{-1}(\tau a)).
\end{align*}  From the definition of $\Gamma$ we see that there is an index $1 \leq i \leq q$ such that $\sigma_{i} =\tau$ and, hence, this product contains the factor $g_{j}(a)$ for some $1 \leq j \leq r$.   From this we see that $\left(\sigma.(g_{1}^{m_{1}}g_{2}^{m_{2}}\dots g_{r}^{m_{r}}) \right)(b)=0$ and so $b$ is contained in the right hand side.

On the other hand, let $b$ be an element of the right hand set.  If $b$ is not an element of the left hand set, then  $\sigma_{i} b$ is not an element of $V(g_{1}, \dotsc , g_{r})$ for any $\sigma_{i} \in \Sigma_{m}$ and, hence, there exists a polynomial $g_{\sigma_{i}} \in \left\{g_{1}, \dotsc , g_{r} \right\}$  such that $g_{\sigma_{i}}(\sigma_{i} b) \neq  0$.  If we set $\sigma = (\sigma_{1}^{-1}, \dotsc , \sigma_{q}^{-1})$ and consider the monomial $g_{\sigma_{1}}\dotsb g_{\sigma_{q}}$, then 
\[
(\sigma.(g_{\sigma_{1}}\dotsb g_{\sigma_{q}}))(b) \neq 0.
\]  This then contradicts our choice of $b$ since, up to reordering factors, $\sigma.(g_{\sigma_{1}}\dotsb g_{\sigma_{q}}) \in Y$.

We can now reduce to a finite set of $\Sigma_{m}T$-invariant polynomials as follows.  For each $\Sigma_{m}$-orbit in $Y$, say $\left\{y_{1}, \dotsc , y_{s} \right\}$, let $\tilde{y}_{1}, \dotsc , \tilde{y}_{s}$ be the elementary symmetric polynomials in $y_{1}, \dotsc , y_{s}$.  
One can then verify that 
\[
Z(y_{1}, \dotsc , y_{s}) = Z(\tilde{y}_{1}, \dotsc , \tilde{y}_{s}).
\]   If we  let $\tilde{Y}$  be the set of all such elementary symmetric polynomials obtained by ranging over the $\Sigma_{m}$-orbits in $Y$, we then have 
\[
\Sigma_{m}Z(g_{1}, \dotsc , g_{r}) = Z(Y)=Z(\tilde{Y}),
\] where $\tilde{Y}$ is a finite set of $\Sigma_{m}T$-invariant polynomials.

We now show that we can realize this variety using Carlson's $L_{\zeta}$-modules. Say $\tilde{Y}=\left\{q_{1}, \dotsc , q_{u} \right\}$.   From \cite[Theorem 4.1.1]{BKN1} we have a ring isomorphism induced by the restriction map:
\[
\HH^{\bullet}(\fg,\fg_{\0};\C)\cong 
\HH^{\bullet}(\ff,\ff_{\0};\C)^{N}\cong S^{\bullet}(\ff_{\1})^{N}.
\]
Therefore, for each $i=1,2,\dots, u$ there exists $\zeta_{i}\in \HH^{n_{i}}(\fg,\fg_{\0};\C)$ corresponding to $q_{i}$ 
under this isomorphism.  Corresponding to this cohomology class is a $\fg$-module homomorphism $\Omega^{n_{i}}(\C )\rightarrow \C $.  Let $L_{{\zeta}_{i}}$ be the kernel of this map; that is, the Carlson module for $\zeta_{i}$.

When we restrict $L_{\zeta_{i}}$ to $\ff$ it will decompose as an $\ff$-module into the direct sum of the 
Carlson module associated with the cohomology class $q_{i}$ in $\HH^{\bullet}(\ff,\ff_{\0};\C)^{N}\subseteq \HH^{\bullet}(\ff,\ff_{\0};\C)$ and a projective $\ff$-module. Standard arguments (e.g.\ see \cite[Lemma 8.5.1]{BaKn} for the Lie superalgebra version)   show that
$$V^{r}_{\ff_{\1}}(L_{\zeta_{i}})=Z(q_{i}).$$  
Now applying the tensor product property we obtain
\[
\Sigma_{m}Z(g_{1}, \dotsc , g_{r}) = Z\left(Y \right) = Z (\tilde{Y} ) = Z\left(q_{1}, \dotsc , q_{u} \right) = V^{r}_{\ff_{\1}}\left( L_{\zeta_{1}} \otimes \dotsb \otimes L_{\zeta_{u}} \right). 
\]  This proves the desired result.
\end{proof}  

%The case when  \texorpdfstring{$t=s=p $}{t=s=p}

\subsection{}\label{SS:simples}  We now calculate the $\f$-variety for the simple modules of $\gl (m|n)$ and thereby achieve realization for $\Sigma_{m}V(p,p,p)$ for any $p=0, \dotsc , m$.  Before stating the theorem, recall that given a simple $\gl (m|n)$-module $L$ 
one can assign an integer to $L$ from among $0,\dotsc ,m$ called the \emph{atypicality of $L$}.  See, for example, \cite[Section 2.4]{BKN4} for the precise definition.  All we need to know here is that there are simple modules of every possible atypicality.

\begin{theorem}\label{T:simples} If $L$ is a simple $\fg = \gl (m|n)$ module of atypicality $\ell$, then 
\[
V^{r}_{\ff_{\1}}\left(L \right) = \Sigma_{m}V(m-\ell,m-\ell,m-\ell).
\]

\end{theorem}

\begin{proof}  From the proof of \cite[Theorem 9.2.1]{BKN4} we have that $V(m-\ell,m-\ell,m-\ell)$ is a subset of $V^{r}_{\ff_{\1}}\left(L \right)$.  Since $V^{r}_{\ff_{\1}}(L)$ is stable under the action of $\Sigma_{m}$ it follows that $\Sigma_{m}V(m-\ell,m-\ell,m-\ell) \subseteq V^{r}_{\ff_{\1}}(L)$.   If $\ell=0$ then $L$ is projective and both sides are the zero variety.  If $\ell=m$, then both sides are equal to $\ff_{\1}$.  In either case the theorem holds and so for the reverse containment it suffices to consider the case when $0 < \ell < m$.

We now consider the reverse containment. Let  
\[
z=\sum_{j=1}^{m} x_{j}e_{m-j+1,m+j} + y_{j} e_{m+j,m-j+1} \in\V^{r}_{\ff_{\1}}\left(L \right).
\] We claim that $z \in \Sigma_{m}V(m-\ell,m-\ell,m-\ell)$.  Using the action of $\Sigma_{m}$ we may assume without loss that there is $1 \leq u \leq v \leq w \leq m$ such that $x_{i}\neq 0$ and $y_{i}\neq 0$ for $i=1, \dotsc , u$, $x_{i}=0$ and $y_{i}\neq 0$ for $i=u+1, \dotsc , v$, $x_{i}\neq 0$ and $y_{i}= 0$ for $i=v+1, \dotsc , w$, and $x_{i}=0$ and $y_{i}= 0$ for $i=w+1, \dotsc , m$.  In particular, $z$ has $m-w$ matching pairs of coordinates which are zero. 

For any fixed non-zero $a \in \R$ let $g_{a} \in G_{\0}$ be the group element given by the diagonal matrix 
\[
g_{a} = \operatorname{Diag}\left(d_{1}, \dotsc , d_{m+n} \right)
\] with 
\[
d_{m-i+1} = \begin{cases} a^{-1}, & i=1, \dotsc , u; \\
                              a, & i=u+1, \dotsc , v; \\
                              a^{-1}, & i=v+1, \dotsc , w; \\
                              1, & i=w+1, \dotsc , m; \\
                              1, & i=1-n, \dotsc , 0; 
\end{cases}. 
\] Since $V^{r}_{\ff_{\1}}\left(L \right)$ is a conical $N$-stable variety we have $ag_{a}zg_{a}^{-1} \in V^{r}_{\ff_{\1}}\left(L \right)$.  That is,
\[
 \sum_{j=1}^{u} \left( x_{j}e_{m-j+1,m+j} + a^{2}y_{j} e_{m+j,m-j+1} \right)+ \sum_{j=u+1}^{v} y_{j} e_{m+j,m-j+1} + \sum_{j=v+1}^{w}x_{j}e_{m-j+1,m+j} \in V^{r}_{\ff_{\1}}\left(L \right).
\]  Iterating, we have 
\[
\sum_{j=1}^{u} \left(x_{j}e_{m-j+1,m+j} + a^{2t}y_{j} e_{m+j,m-j+1} \right)+ \sum_{j=u+1}^{v} y_{j} e_{m+j,m-j+1} + \sum_{j=v+1}^{w}x_{j}e_{m-j+1,m+j}  \in V^{r}_{\ff_{\1}}\left(L \right)
\] for all integers $t>0$.  By taking $0 < a < 1$ and using that $V^{r}_{\ff_{\1}}\left(L \right)$ is a closed variety we see that in the limit we have 
\[
z':=\sum_{j=1}^{u}x_{j}e_{m-j+1,m+j} + \sum_{j=u+1}^{v} y_{j} e_{m+j,m-j+1} + \sum_{j=v+1}^{w}x_{j}e_{m-j+1,m+j} \in V^{r}_{\ff_{\1}}\left(L \right).
\]

However, $[z',z']=0$ and so $z'$ lies in the associated variety for $L$ defined by Duflo-Serganova in \cite{dufloserganova}.  But \cite[Theorem 5.3]{dufloserganova} then implies that $w \leq \ell$.  That is, $z$ has $m-w \geq m-\ell$ matching pairs of coordinates which are zero and, hence, $z \in \Sigma_{m}V(m-\ell,m-\ell,m-\ell)$.  
\end{proof}

\subsection{Lower Bound for \texorpdfstring{$V^{r}_{\ff_{\1}}(H^{0}(G/P, L_{\fp}(\lambda)^{*})^{*})$} {V\textasciicircum r\_{(f\_1)}(H\textasciicircum 0 (G/P, L\_p(lambda)*)*)}} Let $\fg =\mathfrak{gl}(m|n)$ with $m \leq n$, and fix $n-m \leq k \leq n$. 
As in \cref{SS:BorelWeilBott}, let $P=P_{k}$ be the parabolic subgroup of block upper triangular matrices which has Levi subgroup isomorphic to $\operatorname{GL}(m|n-k) \times \operatorname{GL}(k)$.  In this section we determine a lower bound for the $\ff_{\1}$-rank variety of  $H^{0}(G/P, L_{\fp}(\lambda)^{*})^{*}$ when  $\lambda \in X_{0}^{+}$ satisfies certain additional conditions. 

It is convenient to define certain subspaces of $\fg$.  Set 
\[
I = \left\{1, \dotsc , m+k-n, m+n-k+1, \dotsc , m+n  \right\}.
\]
  Let $\fa$ be the subspace spanned by the matrix units $e_{i,j}$ with $i,j \in I$.  Let $\fb$ be the subspace spanned by 
the matrix units $e_{i,j}$ with at least one of $i$ or $j$ not in $I$.  Then 
\begin{equation}\label{E:abdecomp}
\fg  = \fa \oplus \fb 
\end{equation}
as vector spaces and $\fa$ is the Lie subsuperalgebra isomorphic to $\gl(m+k-n|k)$ in the ``corners'' of $\fg$.  Note that our assumption that $n-m \leq k \leq n$ ensures that $0 \leq m+k-n \leq m$.  

% A key property of this decomposition is the fact that a direct calculation shows that \eqref{E:abdecomp} is a decomposition of $\fg$ as a module under the adjoint action for both $\langle x_{0} \rangle$ and $\langle x_{1} \rangle$. 

Recall the parabolic $\Z$-grading defined in Section~\ref{SS:gradings} and that if $\gamma$ is an integral weight, then we write $\deg(\gamma)$ for the degree of $\gamma$ in the parabolic grading.   For example, 
in the parabolic grading the module $H^{0}(G/P, L_{\fp}(\lambda)^{*})^{*}$ has $\deg (\lambda)$ as its highest non-zero degree.  From this it follows that $\fa \cap \fg_{1} \subseteq \fp_{1}$ acts trivially on $[H^{0}(G/P, L_{\fp}(\lambda)^{*})^{*}]_{\deg(\lambda)} $.

Also, note that when $x \in \ff_{\1}\cap \fp$ is written as a sum according to the parabolic grading we have 
\[
x = x_{0}+x_{1},
\]   with $x_{0}$ and $x_{1}$ have degrees $0$ and $1$, respectively.

\begin{lemma}\label{L:Kacgenerator}  Let  $n-m \leq k \leq n$ and $\lambda \in X_{0}^{+}$ such that $\lambda_{i}=-\lambda_{2m+1-i}$ for $i=1,\dotsc , m$.   Let $x \in \ff_{\1} \cap \fp$.  Then 
\[
[H^{0}(G/P, L_{\fp}(\lambda)^{*})^{*}]_{\deg(\lambda)}  \cong Q \oplus U'
\] where the decomposition is as $\fa_{\0}$- and $\langle x_{0} \rangle$-modules. 
 
 Furthermore, 
\[
Q \cong L_{\mathfrak{gl}(m+k-n)}(\lambda_{1}, \dotsc , \lambda_{m+k-n}) \boxtimes L_{\mathfrak{gl}(k)}(\lambda_{m+n-k+1}, \dotsc , \lambda_{m+n})
\]
as an $\fa_{\0} \cong \gl (m+k-n) \oplus \gl(k)$-module and $Q$ is a direct sum of trivial modules for $\langle x_{0} \rangle$.

\end{lemma}

\begin{proof} We first introduce a refinement of the parabolic grading.
Let $a$, $b$, and $c$ be non-negative real numbers and set
\[
h = \operatorname{diag}(a, \dotsc , a, b, \dotsc , b, c, \dotsc ,c) \in T
\]
be the diagonal matrix with $(m+k-n)$ $a$'s, $(2n-2k)$ $b$'s, and $k$ $c$'s.  The action of $h$ on a weight module for $T$ provides a $\Z \times \Z \times \Z$-grading. Namely, given $\gamma = \sum_{i=1}^{m+n}\gamma_{i}\varepsilon_{i}$, if we write $(\deg_{a}(\gamma), \deg_{b}(\gamma), \deg_{c}(\gamma))$ for the degree of an integral weight $\gamma$ with respect to this grading, then $\deg_{a}(\gamma) =\sum_{i=1}^{m+k-n} \gamma_{i}$, $\deg_{b}(\gamma) = \sum_{i=m+k-n+1}^{m+n-k}\gamma_{i}$, and $\deg_{c}(\gamma) = \sum_{i=m+n-k+1}^{m+n}\gamma_{i}$.  Note that by setting $a=b=c^{-1}$ we recover the parabolic grading and that for any weight module $M$ and integral weight $\gamma$ we have $M_{(\deg_{a}(\gamma), \deg_{b}(\gamma), \deg_{c}(\gamma))} \subseteq M_{\deg(\gamma)}$ where $\deg (\gamma) = \deg_{a}(\gamma) + \deg_{b}(\gamma) - \deg_{c}(\gamma)$.

Using this grading (for suitably generic $a, b, c$) we see that under the adjoint action
\[
\fg_{(0,0,0)} \cong \gl(m+k-n) \oplus \gl(n-k|n-k) \oplus \gl(k). 
\] Let $M=H^{0}(G/P, L_{\fp}(\lambda)^{*})^{*}$.  Recall from \eqref{E:parabolicquotientmap} that $M$ is a quotient of the parabolic Verma module.  From this we have that, in the parabolic grading, the degree $0$ part of $M$ is isomorphic to $L_{\fp}(\lambda)$.  This along with a calculation using the $\Z \times \Z \times \Z$-grading shows that
\begin{multline}\label{E:tensor}
M_{(\deg_{a}(\lambda),  \deg_{b}(\lambda), \deg_{c}(\lambda))}   
      \cong L_{\mathfrak{gl}(m+k-n)}(\lambda_{1}, \dotsc , \lambda_{m+k-n}) \boxtimes \\L_{\mathfrak{gl}(n-k|n-k)}(\lambda_{m+k-n+1}, \dotsc , \lambda_{m+n-k}) \boxtimes L_{\mathfrak{gl}(k)}(\lambda_{m+n-k+1}, \dotsc , \lambda_{m+n})
\end{multline}
as $\fg_{(0,0,0)}$-modules.

Observe that $x_{0}$ is an element of the detecting subalgebra of $\mathfrak{gl}(n-k|n-k) \subseteq \fg_{(0,0,0)}$.    Our assumption on $\lambda$ implies that $ L_{\mathfrak{gl}(n-k|n-k)}(\lambda_{m+k-n+1}, \dotsc , \lambda_{m+n-k})$ is a simple $\mathfrak{gl}(n-k|n-k)$-module with atypicality $n-k$ (the maximum possible).  Theorem~\ref{T:simples} then implies that $ L_{\mathfrak{gl}(n-k|n-k)}(\lambda_{m+k-n+1}, \dotsc , \lambda_{m+n-k})$ is not projective as a $U(\langle x_{0} \rangle)$-module.  That is, we can find a non-zero vector 
\[
v_{0} \in  L_{\mathfrak{gl}(n-k|n-k)}(\lambda_{m+k-n+1}, \dotsc , \lambda_{m+n-k})
\]
such that $v_{0}$ spans an $\langle x_{0} \rangle$-trivial direct summand and
\begin{equation}\label{E:splitter}
L_{\mathfrak{gl}(n-k|n-k)}(\lambda_{m+k-n+1}, \dotsc , \lambda_{m+n-k}) = \C v_{0} \oplus J
\end{equation}
as $\langle x_{0} \rangle$-modules.  

Set 
\begin{align*}
A &= L_{\mathfrak{gl}(m+k-n)}(\lambda_{1}, \dotsc , \lambda_{m+k-n}) \boxtimes \C v_{0} \boxtimes L_{\mathfrak{gl}(k)}(\lambda_{m+n-k+1}, \dotsc , \lambda_{m+n}) \\
B &= L_{\mathfrak{gl}(m+k-n)}(\lambda_{1}, \dotsc , \lambda_{m+k-n}) \boxtimes J \boxtimes L_{\mathfrak{gl}(k)}(\lambda_{m+n-k+1}, \dotsc , \lambda_{m+n})
\end{align*}
Combining \eqref{E:splitter} with \eqref{E:tensor} we obtain
\begin{align}\label{E:breaker}
M_{(\deg_{a}(\lambda), \deg_{b}(\lambda), \deg_{c}(\lambda))} \cong A \oplus B
\end{align} as $\fa_{\0}$ and $\langle x_{0} \rangle$-modules.  Now set $Q$ equal to $A$ and $U'$ equal to the sum of $B$ and all trigraded components of $M_{\deg (\lambda)}$ excluding $M_{(\deg_{a}(\lambda), \deg_{b}(\lambda), \deg_{c}(\lambda))}$.  Then $M_{\deg (\lambda)} \cong Q\oplus U'$ as vector spaces.  The fact that $x_{0}$ and $\fa_{\0}$ are of degree $(0,0,0)$ along with our choice of $A$ and $B$ makes it straightforward to see that this decomposition has the asserted properties.  
\end{proof}

Let $\fu^{-}$ be the opposite nilradical for our fixed parabolic subalgebra $\fp$. Set $\fa^{-} = \fa \cap \fu^{-}$ and $\fb^{-}= \fb \cap  \fu^{-}$.  In the PBW basis for $U(\fu^{-})$ given by all monomials in the matrix units, set $S$ to be the subspace of $U(\fu^{-})$ spanned by all monomials containing at least one matrix unit from $\fb^{-}$. Since $\fu^{-} = \fa^{-} \oplus \fb^{-}$ the PBW theorem for Lie superalgebras implies that the following decomposition holds as vector spaces:
\begin{equation}\label{E:symmetricdecomposition}
U(\fu^{-}) \cong U(\fa^{-}) \oplus S.
\end{equation}
We can now provide a lower bound for the $\ff_{\1}$-rank variety of $H^{0}(G/P, L_{\fp}(\lambda)^{*})^{*}$.  
%\[
%S= \bigoplus_{i \geq 0, j >0} S_{\text{super}}^{i}(\fa^{-}) \otimes S_{\text{super}}^{j}(\fb^{-}).
%\]  Here and below we write $S_{\text{super}}^{d}(\fa^{-})$ for the $d$th supersymmetric power of $\fa^{-}$. 

%\begin{equation}\label{E:symmetricdecomposition}}
%\bigoplus_{t\geq 0}} S_{\text{super}}^{t}(\fu^{-}) \cong \bigoplus_{t \geq 0} \bigoplus_{i+j=t} S_{\text{super}}^{i}(\fa^{-}) \otimes S_{\text{super}}^{j}(\fb^{-}) \notag% 
%    = \left(\bigoplus_{t=0}^{k^{2}}  S_{\text{super}}^{t}(\fa^{-}) \right) \bigoplus S.
%\end{equation}

\begin{theorem}\label{T:Nprimesupport}  Let  $n-m \leq k \leq n$ and let $\lambda \in X_{0}^{+}$ be a dominant integral weight such that $\lambda_{i}=-\lambda_{2m+1-i}$ for $i=1,\dotsc , m$.  Furthermore, assume $\lambda_{i}-\lambda_{i+1}>(m+k-n)k+mn$ for $i=1, \dotsc , m-1, m+1, \dotsc , m+n$.  Then 
\begin{equation}
\ff_{\1} \cap  \fp  \subseteq V^{r}_{\ff_{\1}}\left( H^{0}(G/P, L_{\fp}(\lambda)^{*})^{*}\right).
\end{equation}
\end{theorem}

\begin{proof}  Let $x \in \ff_{\1} \cap \fp$.  From the rank variety description it suffices to show that $M:=H^{0}(G/P, L_{\fp}(\lambda)^{*})^{*}$ is not projective as a $U(\langle x \rangle)$-module.  By the previous lemma we may write 
\[
 M_{\deg (\lambda)} \cong  Q \oplus U.
\]

Consider $U(\fa )Q \subseteq M$.  With respect to $\fa$, $U(\fa)Q$ is generated by a highest weight vector of  weight $\gamma:=(\lambda_{1}, \dotsc , \lambda_{m+k-n} | \lambda_{m+n-k+1}, \dotsc , \lambda_{m+n} )$ and 
there exists a surjective $\fa$-homomorphism from the Kac module $K^{+}_{\mathfrak{gl}(m+k-n|k)}(\gamma)$ onto $U(\fa)Q$.  The key observation is that Theorem~\ref{T:grading} (with $d=(m+k-n)k$) implies that this map is also injective.  Thus 
\[
U(\fa)Q \cong K^{+}_{\mathfrak{gl}(m+k-n|k)}(\gamma)
\]
as $\fa$-modules.  Furthermore, our assumption on $\lambda$ ensures that $\gamma$ has atypicality $m+k-n$ (the maximum possible) for $\fa$.

Since $M = U(\fu^{-})M_{\deg (\lambda)}$, we can apply \cref{T:grading}, \cref{L:Kacgenerator}, and \cref{E:symmetricdecomposition} to see that we have a vector space decomposition: 
\begin{equation}\label{E:breakdown}
M = U(\fa)Q \oplus U(\fa )U \oplus \left( SM_{\deg (\lambda)} \right).
\end{equation}

We now consider this decomposition under the action of $x_{0}$ and $x_{1}$.  A direct calculation shows that for $i=0,1$ we have $[x_{i}, \fa ] \subseteq \fa$ and $[x_{i}, \fb ] \subseteq \fb$.  This plus the fact that $Q$, $U$, $M_{\deg (\lambda)}$ are $\langle x_{i} \rangle$-modules implies this is a decomposition into $\langle x_{i} \rangle$-modules and, hence, into $\langle x \rangle$-modules.

Since $U(\fa)Q$ is isomorphic to a Kac module of maximal atypicality for $\fa$, and $x_{1}$ lies in the detecting subalgebra of $\fa$, it follows from the proof of \cite[Theorem 6.2.1]{BKN4} that $U(\fa)Q$ has a direct sum decomposition as an $\langle x_{1} \rangle$-module with one direct summand being trivial.  Furthermore, since $[x_{0},\fa ] = 0$ and $x_{0}$ acts as zero on $Q$, we have that $\langle x_{0} \rangle$ acts trivially on $U(\fa)Q$.  From this it follows that the decomposition of $U(\fa)Q$ is also as $\langle x_{0} \rangle$-modules and, hence, as $\langle x \rangle$-modules.  That is, $U(\fa)Q$ contains a trivial direct summand as an $\langle x \rangle$-module and hence, by \eqref{E:breakdown}, so does $M$.  Therefore, $M$ is not projective as a $U(\langle x \rangle)$-module and so $x \in V^{r}_{\ff_{\1}}(M)$ as desired.
\end{proof}

\subsection{}					

 Recall from \cref{ex:supportsforstandardparabolics} that $\fg = \gl (m|n)$ admits a $\Z$-grading $\fg=\fg_{-1} \oplus\fg_{0}\oplus \fg_{1}$.   We continue the setup of \cref{SS:BorelWeilBott} with $G=\operatorname{GL}(m|n)$ and with $G_{-1}$  as the closed subgroup of $G$ with $\operatorname{Lie}(G_{-1})=\fg_{-1}$.  In particular, the parabolic subgroup $P \subseteq G$ consists of block upper triangular matrices with Levi subgroup $L=\operatorname{GL}(m|n-k) \times \operatorname{GL}(k)$ for some $0 \leq k \leq n$.   Then $P=L\ltimes U$, where $L$ is the Levi subgroup and $U$ is a unipotent subgroup. Let $Z$ be the closed subgroup $P\cap G_{-1}$ and $\mathfrak z$ be the Lie superalgebra associated to $Z$.   

  The following theorem provides a relationship between supports for a $P$-module $N$ and the $G$-module $\Ind_{P}^{G} N$ (cf.\ \cite[(5.4.1) Theorem]{NPV}).  Note that $\fz$ and $\fg_{-1}$ are abelian Lie superalgebras concentrated in odd degree and, hence, the support data $V_{\fz}(-)$ and $V_{\fg_{-1}}(-)$ are as in \cref{ex:abelianLiesuper} and \cref{ex:supportsforstandardparabolics}.  
\begin{theorem} \label{T:supportupperbound} Let $N$ be a finite dimensional $P$-module and $M=\operatorname{ind}_{P}^{G}N$. Suppose that $\HH^{j}(G/P, N)=R^{j}\operatorname{ind}_{P}^{G}N=0$ for all $j>0$. 
Then 
\begin{itemize} 
\item[(a)] $V_{\fg_{-1}}(\operatorname{ind}_{P}^{G} N)=G_{0}\cdot V_\fz(M,N)$. 
\item[(b)] $V_{\fg_{-1}}(\operatorname{ind}_{P}^{G} N)\subseteq G_{0}\cdot  V_\fz(N)$. 
\end{itemize} 
\end{theorem} 

\begin{proof} Part (b) follows from part (a) since $V_\fz(M,N)\subseteq  V_\fz(N)$. This inclusion 
also shows that 
$$V_\fz(M,N)\subseteq  V_\fz(M)\subseteq V_{\fg_{-1}}(M).$$
Since $V_{\fg_{-1}}(M)$ is $G_{0}$-invariant, it follows that 
$$G_{0} \cdot V_\fz(M,N)\subseteq  V_{\fg_{-1}}(M).$$

We now consider the spectral sequence (\ref{E:inversionspectral}).  One 
can apply \cite[(5.2.1) Proposition]{NPV} to show that the ideal ${\mathcal I}$ in $R=\operatorname{H}^{\bullet}(\fg_{-1},\C)$  which defines 
$G_{0} \cdot \VV_\fz(M,N)$ annihilates $E_{2}^{p,q}$ for all $p, q\geq 0$. 

Therefore, ${\mathcal I}\subseteq \text{Ann}_R\ E^{\bullet,\bullet}_\infty$.
From the Grothendieck vanishing theorem, $R^j\Ind_{P_{0}}^{G_{0}} N=0$ for $j> s=\dim\,G_{0}/P_{0}$. 
Consequently,  $\Ext^\bullet_{\fg_{-1}}(\Ind_P^G N,\Ind_P^G N)$ has an $R$-stable filtration 
$0=\FF^0\subseteq \FF^1\subseteq \cdots \subseteq \FF^s=\Ext^\bullet_{\fg_{-1}}(\Ind_P^G N, \Ind_P^G N)$
where $\FF^i/\FF^{i-1}=\bigoplus_j E^{s-i+1,j}_\infty$.

Observe that ${\mathcal I}\cdot \FF^i\subseteq \FF^{i-1}$, thus ${\mathcal I}^s$ annihilates 
$\Ext^\bullet_{\fg_{-1}}(\Ind_P^G N,\Ind_P^G N)$ which proves that  $V_{\fg_{-1}}(\Ind_P^G N)\subseteq
G_{0}\cdot V_{\fz}(\Ind_P^G N,N)$. 
\end{proof}

\subsection{} The previous theorem allows us to give us a constraint on the $\ff_{\1}$ rank variety of $H^{0}(G/P, N)$ for certain $P$-modules $N$.  
%Before stating the result we recall notation introduced in \cref{SS:varietynotations}.  Let $s, t, k$ be nonnegative integers with $m\geq s, t \geq k \geq 0$ and, as in \cref{E:vanishingcoordinates}, set 
%\begin{equation*}
%V(s,t,k) = Z(X_{1},\dots,X_{s},Y_{s-k+1},\dots,Y_{s-k+t} ).
%\end{equation*}

\begin{theorem}\label{T:inducedmodulevarietyfirstinclusion}  Fix $n-m \leq k \leq n$ and let $P$ be the parabolic defined above. Let $N$ be a finite dimensional $P$-module such that $R^{j}\operatorname{ind}_{P}^{G}N=0$ for $j>0$. Then 
\[
V^{r}_{\ff_{\1}}(H^{0}(G/P, N)) \subseteq \Sigma_{m}V(0,m-(n-k),0).
\] 

\end{theorem}
 
\begin{proof}    Let  
\[
z \in V^{r}_{\ff_{\1}}(H^{0}(G/P,N))
\] and write $z=z_{-1} + z_{1}$ where $z_{\pm 1} \in \fg_{\pm 1}$.

For any fixed non-zero $a \in \R$ let $g_{a} \in G_{\0}$ be the group element given by the diagonal matrix 
\[
g_{a} = \operatorname{Diag}\left(d_{1}, \dotsc , d_{m+n} \right)
\] with 
\[
d_{i} = \begin{cases} a, & i=1, \dotsc , m, \\
                              a^{-1}, & i=m+1, \dotsc , m+n.
\end{cases}. 
\] Since $V^{r}_{\ff_{\1}}(H^{0}(G/P,N))$ is a conical $N$-stable variety we have $a^{2}g_{a}zg_{a}^{-1} \in V^{r}_{\ff_{\1}}(H^{0}(G/P,N))$.  That is, 
\[
a^{2}g_{ a}zg_{ a}^{-1} = z_{-1} + a^{4}z_{1} \in V^{r}_{\ff_{\1}}(H^{0}(G/P,N))
\]  Iterating, we have 
\[
a^{2t}g^{t}_{ a}zg_{ a}^{-t} = z_{-1} + a^{4t}z_{1} \in V^{r}_{\ff_{\1}}(H^{0}(G/P,N))
\] for all integers $t>0$.  By taking $0 < a < 1$ and using that $V^{r}_{\ff_{\1}}(H^{0}(G/P,N))$ is a closed variety we see that in the limit we have 
\[
z_{-1} \in V^{r}_{\ff_{\1}}(H^{0}(G/P,N)).
\] 

But then $z_{-1} \in V_{{\fg}_{-1}}(H^{0}(G/P,N))$ and so it follows from \cref{T:supportupperbound}(b) that 
\[
z_{-1} \in G_{0}\cdot V_\fz(N) \subseteq G_{0}\cdot \mathfrak{z}.
\]  However, since $\mathfrak{z}$ consists matrices of rank at most $n-k$, so too must the elements of $G_{0}\cdot \mathfrak{z}$.  For $z_{-1}$ to have rank at most $n-k$ it follows that $z \in \Sigma_{m}V(0,m-(n-k),0)$.
\end{proof}

\subsection{The case when  \texorpdfstring{$s=p=0 $}{s=p=0}}\label{SS:casetparezero}  We can now realize $\Sigma_{m}V(0,t,0)$ for any $0\leq t \leq m$. 

%If $t=p=0$ and $s=0$, then $V(s,t,p)=\ff_{\1}$ and we may take $M$ to be the trivial $\fg$-module.  On the other hand, if $t=p=0$ and $s=m$, then $V(s,t,p) = \ff_{\1} \cap \fg_{1}$ and Theorem~6.4.1 and the proof of Theorem~9.2.1 in \cite{BKN4} show that we can take $M=K(0)$, the Kac module of highest weight $0$.  The following theorem handles the intermediate cases. 

\begin{theorem}\label{T:Vstprealizations}  Let $0\leq t \leq m$ and let $\lambda \in X_{0}^{+}$ be a dominant integral weight with $\lambda_{i}=-\lambda_{2m+1-i}$ for $i=1, \dotsc , m$ and $\lambda_{i}-\lambda_{i+1} > t(n-m+t) + mn$ for $i=1, \dotsc ,$\linebreak[1] $ m-1, m+1, \dotsc , m+n-1$.  Let $P$ be the parabolic subgroup of $\operatorname{GL}(m|n)$ which is block upper triangular and has Levi subgroup isomorphic to $\operatorname{GL}(m|m-t) \times \operatorname{GL}(n-m+t)$.  Then 
\[
V^{r}_{\ff_{\1}}\left(H^{0}(G/P, L_{\fp}(\lambda)^{*})^{*} \right) = \Sigma_{m}V(0,t,0).
\]
\end{theorem}

\begin{proof}   Set $M=H^{0}(G/P, L_{\fp}(\lambda)^{*})$.  Since the higher derived functors vanish by Theorem~\ref{T:SuperBWB} and since $V^{r}_{\ff_{\1}}(M^{*})=V^{r}_{\ff_{\1}}(M)$, we can use Theorem~\ref{T:inducedmodulevarietyfirstinclusion} to obtain the inclusion 
\[
V^{r}_{\ff_{\1}}\left(H^{0}(G/P, L_{\fp}(\lambda)^{*})^{*} \right) \subseteq \Sigma_{m}V(0,t,0).
\]

On the other hand, if we set $k=n-m+t$ in Theorem~\ref{T:Nprimesupport}, then we have the inclusion 
\[
\ff_{\1} \cap  \fp  \subseteq V^{r}_{\ff_{\1}}(M). 
\] But since $V(0,t,0)= \ff_{\1}\cap \fp$ and $V^{r}_{\ff_{\1}}(M)$ is stable under the action of $\Sigma_{m}$, we then have  
\[
\Sigma_{m}V(0,t,0) \subseteq V^{r}_{\ff_{\1}}(M).
\]  This proves the desired equality.
\end{proof}

\subsection{The case when  \texorpdfstring{$t=p=0$}{t=p=0}}\label{SS:casesparezero}  If we set $\bar{i} \in \Z_{2}$ by the rule $\bar{i}=\0 $ for $i=1, \dotsc , m$ and $\bar{i}=\1 $ for $i=m+1, \dotsc , m+n$, then we can define an automorphism of $\gl (m|n)$  by $e_{i,j} \mapsto -(-1)^{\bar{i}(\bar{i}+\bar{j})}e_{j,i}$.  Given a finite dimensional $\fg$-module $M$ let $M^{\tau}$ be the $\fg$-module $M^{*}$ with the action twisted by this automorphism.  

\begin{theorem}\label{T:Vstprealizations2}  If $0 \leq  s \leq m$, then there exists a finite dimensional $\fg$-module $M$ such that 
\[
V^{r}_{\ff_{\1}}(M) = \Sigma_{m}V(s,0,0).
\]
\end{theorem}
 
\begin{proof} This follows from \cref{T:Vstprealizations} and the observation that for any finite dimensional $\fg$-module $M$ we have the identity 
\[
V^{r}_{\ff_{\1}}(M^{\tau}) = \tau \left(V^{r}_{\ff_{\1}}(M) \right).\qedhere
\]
\end{proof}

\subsection{The case when  \texorpdfstring{$s,t \geq 0$}{s, t >= 0} and  \texorpdfstring{$p=0$}{p=0}} \label{SS:Coordinate} In what follows we make extensive use of the fact that $\ff$-support varieties are a support data as defined in Section~\ref{SS:supportdata}.  We also use the fact that if $G$ is a group acting on a set and $A$, $B$ are subsets of that set  then  
\[
GA \cap GB = G(A\cap GB) = G(GA \cap B).
\]

\begin{prop} \label{P:CoordsNoOverlap}
Fix $s, t \ge  0$. The variety $\Sigma_{m}V(s,t,0)$ is realizable as the support of a module in $\F$.
\end{prop}

\begin{proof}
The proof is by induction on $t$. The base case, $t=0$, is Theorem~\ref{T:Vstprealizations2}. So let $t>0$ and assume the result is true for $\Sigma_{m} V(s,t-1,0)$. Consider first
$$
U := \Sigma_{m} V(s,t-1,0) \cap \Sigma_{m} V(0,t,0).
$$
This variety is realizable as the support of a tensor product, since the first term on the right is realizable by the induction hypothesis, and the second by Theorem~\ref{T:Vstprealizations2}. Viewing the intersection as imposing the condition of the second variety on the first, we need one more $Y_{j}$ to vanish, and this can either overlap with an $X_{i}$ that is already required to vanish, or not. Thus
$$
U = \Sigma_{m} V(s,t,0) \cup \Sigma_{m} V(s,t,1).
$$

For clarity in what follows we will write $Z_{i}=X_{i}Y_{i}$ for $i=1, \dotsc , m$.  Write $p_{j}$ for the degree $j$ elementary symmetric polynomial in the variables $Z_{1}, \dotsc , Z_{m}$, for $j=1, \dotsc , m$.

Next consider
$$
W := \Sigma_{m} V(s,t-1,0) \cap Z(p_{m-s-(t-1)}) .
$$
Again, $W$ is realizable as the support of a tensor product, by the induction hypothesis and \cref{SS:Weight0}.  Note that on $V(s,t-1,0)$, the only term of $p_{m-s-(t-1)}$ which does not automatically vanish is the product of the last $m-s-(t-1)$ variables $Z_{i}$. Thus $\Sigma_{m} V(s,t-1,0) \cap Z(p_{m-s-(t-1)})$ is the variety defined by the vanishing of $s$ of the $X_{i}$, $t-1$ of the $Y_{i}$, and either one additional $Y_{i}$ or one additional $X_{i}$. That is,
\begin{equation*}
W =  \Sigma_{m}V(s,t,0) \cup \Sigma_{m} V(s+1,t-1,0).
\end{equation*}
Finally, $U\cap W$ is realizable as the support of a tensor product, and a calculation of set intersections shows that 
\[
U \cap W = V(s,t,0).\qedhere
\] 
%\begin{align*}
%U\cap W &= \Sigma_{m}V(s,t,0)\cap \big(\Sigma_{m}V(s,t,1) \cap \Sigma_{m}V(s+1,t-1,0)\big) \\
%    &= \Sigma_{m} V(s,t,0)\cup \Sigma_{m}V(s+1,t,1) \\
%    &= \Sigma_{m}V(s,t,0).
%\end{align*}
%In the last step we have used the observation that $V(s+1,t,1)$ is the subset of $V(s,t,0)$ obtained by setting $X_{s+1}=0$.
\end{proof}

\subsection{The case when  \texorpdfstring{$s,t \geq p>0$}{s, t >= p>0}} \label{SS:CoordinateII} 

\begin{prop} \label{P:CoordsOverlap}
Fix $s, t \ge p > 0$. The variety $\Sigma_{m} V(s,t,p)$ is realizable as the support of a module in $\F$.
\end{prop}

\begin{proof}
First consider the variety $V(p,p,p)$ where $p>0$. Recall from Theorem~\ref{T:simples} that $V^{r}_{\ff_{\1}}(L(\la)) = \Sigma_{m}\cdot V(p,p,p)$ for a simple module $L(\la)$ of highest weight $\la$ having atypicality $m-p$. 

Now consider the general case $V(s,t,p)$ with $s, t \ge p > 0$. Observe that 
$$\Sigma_{m} V(s,t,p)=\Sigma_{m} V(s,t-p,0) \cap \Sigma_{m} V(0,t,0) \cap \Sigma_{m} V(p,p,p).$$ 
But $\Sigma_{m}V(s,t-p,0)$ is realizable as the support of a module in $\F$ by \cref{P:CoordsNoOverlap}; likewise $\Sigma_{m} V(0,t,0)$ is realizable by \cref{T:Vstprealizations}; and, as we noted above, so is $\Sigma_{m} V(p,p,p)$. Therefore, $\Sigma_{m} V(s,t,p)$ is the support of the tensor product.
\end{proof}

We have now proven that all varieties of the form  $\Sigma_{m} V(s,t,p)$ are realizable.  Any variety $V$ as in \cref{E:generalV} in which $r=0$ is conjugate to some $V(s,t,p)$ under the action of $\Sigma_{m}$.  Therefore, we have in fact realized all varieties of the form $\Sigma_{m}V$ with $V$ as in \cref{E:generalV} and with $r=0$.

\subsection{Mixed Coordinate and Weight Zero Functions} \label{SS:Mixed} In this subsection, we will verify the realization property in the case where $V$ is given by the vanishing of one or more coordinate functions and one or more weight zero polynomials on $\f_{\1}$.

\begin{prop} \label{P:mixed}
Let $V=V(s,t,0) \cap Z(g_{1},\dots,g_{r})$, where $g_{1},\dots,g_{r}$ are weight zero polynomials. Then the variety $\Sigma_{m} V$ is realizable as the support of a module in $\F$.
\end{prop}

\begin{proof}
For $i=1, \dotsc ,m$, set $Z_{i}=X_{i}Y_{i}$ and observe that, since they are weight zero, $g_{1},\dots,g_{r}$ are polynomials in $Z_{1},\dots,Z_{m}$. However, in $V(s,t,0)$ we have $X_{1}=\dots =X_{s}=Y_{s+1}= \dots = Y_{s+t}=0$, so we may drop any terms in the $g_{k}$ that involve $Z_{1},\dots,Z_{s+t}$ without changing $V$, and thus we may assume that the $g_{k}$ are polynomials in $Z_{s+t+1},\dots,Z_{m}$.

Set
$$
U = \Sigma_{m} Z(Z_{1},\dots,Z_{s+t},g_{1},\dots,g_{r} ).
$$
Then $U$ is realizable as the support of a module in $\F$ by \cref{SS:Weight0}, and $\Sigma_{m}V \subseteq U$.   Set $$W=\Sigma_{m} V(s,t,0).$$  Then $W$ is realizable by \cref{P:CoordsNoOverlap} and clearly $\Sigma_{m} V \subset W$. Consequently $U \cap  W$ is realizable by the tensor product property and $\Sigma_{m}V \subseteq U \cap W$.   Therefore it suffices to prove  $$U\cap W \subseteq \Sigma_{m}V.$$

We first observe that since $Z_{i}=0$ if and only if $X_{i}=0 \text{ or } Y_{i}=0$, and since we can permute the indices $1,\dots,s+t$ without affecting the polynomials $g_{1}, \dotsc , g_{r}$, we can write

\begin{align*}
U & =  \Sigma_{m} \left[ \left(\bigcup_{j=0}^{s+t}   V(j,s+t-j,0)  \right)\cap Z( g_{1},\dots,g_{r} ) \right] \\
 & = \bigcup_{j=0}^{s+t} \Sigma_{m} \left(  V(j,s+t-j,0) \cap Z( g_{1},\dots,g_{r} ) \right) =: \bigcup_{j=0}^{s+t} U_{j}.
\end{align*}

%Note that $\Sigma_{m} V=U_{s} \subset U$.

From this we see that to prove the claim it suffices to show $U_{j}\cap W\subset\Sigma_{m} V$ for all $j$.  Assume $j\leq s$; the proof for $j>s$ is similar and is left to the reader.

Write 
$$
U_{j}\cap W =\Sigma_{m} \big( V(j,s+t-j,0) \cap Z( g_{1},\dots,g_{r} ) \cap\Sigma_{m} V(s,t,0) \big).
$$
Viewing the intersection as imposing the conditions of $W$ on $V(j,s+t-j,0)$, we are now required to have $s$ coordinates $X_{i}$ to be zero.  As those in the range $1, \dotsc , j$ are already zero, we must have an additional number, say $u$, in the range $j+1, \dotsc , s+t$ and, say $v$, in the range $s+t+1, \dotsc ,  m$.  We note that then $j+u+v=s$.  We also observe that those which are in the range $1 \leq i \leq s+t$ can have their positions permuted without changing $g_{1}, \dotsc , g_{r}$.  Combining these observations we see that 
\[
U_{j} \cap W \subseteq \bigcup_{u=0}^{s-j} U_{j,u},
\] where 
\begin{equation*} \label{E:Uju}
U_{j,u} :=\Sigma_{m} \big(V(j+u,s+t-j,u) \cap Z( X_{s+t+1},\dots, X_{s+t+v}  ) \cap Z( g_{1},\dots,g_{r} )\big).
\end{equation*} Here we assume for simplicity of notation that the $v$ coordinates which involve indices in the range $s+t+1, \dotsc , m$ are $X_{s+t+1},\dots, X_{s+t+v}$.  An identical argument applies in the general case.

Let $\pi\in\Sigma_{m}$ be the product of the transpositions $(j+u+i, s+t+i)$ for $1\le i\le v$.  Recalling that $g_{1},\dots,g_{r}$ are polynomials in $Z_{s+t+1},\dots,Z_{m}$ we see that while $\pi g_{1}, \dotsc \pi g_{r}$ are not the same polynomials as $g_{1}, \dotsc , g_{r}$, we do have 
\begin{equation} \label{E:Ujuvi}
U_{j,u} =\Sigma_{m} \big(V(j+u,s+t-j,u) \cap Z( X_{s+t+1},\dots, X_{s+t+v} ) \cap Z(\pi g_{1},\dots,\pi g_{r} )\big).
\end{equation}
But now applying $\pi$ to the entire intersection in \cref{E:Ujuvi} and using that $j+u+v=s$, we see that
\[
U_{j,u} \subseteq \Sigma_{m}V.
\] This proves the claim and the theorem.
\end{proof}

\subsection{General Case} \label{SS:generalcase}
Finally we can prove realization in general.  That is, we will have finally verified  (\ref{SS:supportdata}.\ref{E:supporteight}) holds for the support data $\widehat{V}$ introduced in \cref{S:MainTheorem}.

\begin{theorem}\label{T:finalrealization}  Let $m \geq s,t \geq p \geq 0$ and let  $g_{1}, \dotsc , g_{r}$ be homogeneous weight zero polynomials.  Set $V=V(s,t,p) \cap Z (g_{1}, \dotsc , g_{r})$. Then there exists a finite dimensional  $\fg$-module $M$ such that 
\[
V^{r}_{\ff_{\1}}(M) = \Sigma_{m}V.
\]

\end{theorem}

\begin{proof}  Let $U=\Sigma_{m}\left(V(s,t-p,0)\cap Z(g_{1}, \dotsc , g_{r}) \right)$,  $W=\Sigma_{m}V(p,p,p)$,  $Y=\Sigma_{m}V(0,t,0)$.  By \cref{P:mixed,T:Vstprealizations,T:simples}, respectively, these varieties can be realized as the support varieties of finite dimensional $\fg$-modules.  Their tensor product then realizes $U\cap W \cap Y$.  An analysis similar to the proof of the previous result shows that $\Sigma_{m}V=U\cap W \cap Y$.
\end{proof}

\let\section=\oldsection
\bibliographystyle{amsalpha}
\bibliography{BKN5}

\end{document}